\documentclass[a4paper,10pt]{scrartcl}
\usepackage[utf8]{inputenc}
\usepackage[english]{babel}
\usepackage[fixlanguage]{babelbib}
\usepackage{cite}
\usepackage{amssymb, amsmath, amstext, amsopn, amsthm, amscd, amsxtra, amsfonts}
\usepackage{yfonts, mathrsfs}
\usepackage{colonequals}
\usepackage{enumerate}
\usepackage{graphicx}%
\usepackage{colonequals}
\usepackage[all]{xy}
\usepackage{hyperref}%
\hypersetup{
urlcolor = red,
colorlinks = true,
linkcolor = blue,
citecolor = blue,
linktocpage = true,
pdftitle = {Representations of Lie groupoids and infinite-dimensional Lie groups},
pdfauthor = {Habib Amiri, Alexander Schmeding},
bookmarksopen = true,
bookmarksopenlevel = 1,
unicode = true,
hypertexnames =false
}%
\swapnumbers
\newtheoremstyle{standard}
{16pt} 
{16pt} 
{} 
{} 
{\bfseries}
{} 
{ } 
{{\thmname{#1~}}{\thmnumber{#2.}}\thmnote{~(#3)}} 
\newtheoremstyle{kursiv}
{16pt} 
{16pt} 
{\itshape} 
{} 
{\bfseries}
{} 
{ } 
{{\thmname{#1~}}{\thmnumber{#2.}}\thmnote{~(#3)}} 
\theoremstyle{standard}
\newtheorem{defn} [subsection]{Definition}
\newtheorem{ex} [subsection]{Example}

\newtheorem{rem} [subsection]{Remark}

\newtheorem{setup} [subsection]{}

\theoremstyle{kursiv}

\newtheorem{prop} [subsection]{Proposition}
\newtheorem{cor} [subsection]{Corollary}
\newtheorem{lem} [subsection]{Lemma}

\newcommand{\cG}{\ensuremath{\mathcal{G}}}

\newcommand{\dd}{\mathop{}\!\mathrm{d}}

\newcommand{\R}{\ensuremath{\mathbb{R}}}

\newcommand{\N}{\ensuremath{\mathbb{N}}}

\newcommand{\M}{\ensuremath{\mathbb{M}}}
\newcommand{\toto}{\ensuremath{\nobreak\rightrightarrows\nobreak}}
\newcommand{\coloneq}{\colonequals}
\newcommand{\SG}{S_{\cG}}

\newcommand{\RepG}{\phi} 
\newcommand{\ActG}{\Phi} 
\newcommand{\rhoG}{\rho_\RepG} 
\newcommand{\Prop}{\ensuremath{\operatorname{Prop}}}
\newcommand{\Bis}{\ensuremath{\operatorname{Bis}}}
\newcommand{\ev}{\ensuremath{\operatorname{ev}}}

\DeclareMathOperator{\Diff}{Diff}

\DeclareMathOperator{\id}{id}
\DeclareMathOperator{\GL}{\mathrm{GL}}
\DeclareMathOperator{\SL}{{\mathcal{S}l}}
\DeclareMathOperator{\im}{im}
\DeclareMathOperator{\Aut}{Aut}
\newcommand{\CfS}{\ensuremath{C^\infty_{\mathrm{fS}}}}
\newcommand{\Frechet}{Fr\'{e}chet}

\newcommand{\cat}[1]{\ensuremath{\mathsf{\mathop{#1}}}}
\newcommand{\Rep}{\cat{Rep}}

\begin{document}
\title{Linking Lie groupoid representations and representations of infinite-dimensional Lie groups}
 \author{Habib Amiri\footnote{University of Zanjan, Iran
\href{mailto:h.amiri@znu.ac.ir}{h.amiri@znu.ac.ir}}\hspace{2mm} and Alexander
Schmeding\footnote{TU Berlin, Germany
\href{mailto:schmeding@tu-berlin.de}{schmeding@tu-berlin.de}
}%
}
{\let\newpage\relax\maketitle}

\begin{abstract}
The present paper links the representation theory of Lie groupoids and infinite-dimensional Lie groups.
We show that smooth representations of Lie groupoids give rise to smooth representations of associated Lie groups. The groups envisaged here are the bisection group and a group of groupoid self maps. Then representations of the Lie groupoids give rise to representations of the infinite-dimensional Lie groups on spaces of (compactly supported) bundle sections. Endowing the spaces of bundle sections with a fine Whitney type topology, the fine very strong topology, we even obtain continuous and smooth representations. It is known that in the topological category, this correspondence can be reversed for certain topological groupoids. We extend this result to the smooth category under weaker assumptions on the groupoids.
\end{abstract}



\medskip

 (secondary)

\tableofcontents

\section*{Introduction and statement of results} \addcontentsline{toc}{section}{Introduction and statement of results}

Groupoids and their representations appear in a variety of mathematical areas. For example, they have been studied in connection with geometric quantisation \cite{MR2343354}, representations of $C^*$-algebras \cite{MR2844451,MR584266,MR0211351} and the study of covariance of differential operators \cite{MR0438405,MR1958838} to name just a few.
In the present paper we are interested in smooth representations of Lie groupoids. Our aim is to relate the representations of a Lie groupoid to the  representation of certain infinite-dimensional Lie groups associated to the Lie groupoid.
Recall that to every Lie groupoid $\cG$ one can associate its group of bisections $\Bis (\cG)$, which is an infinite-dimensional Lie group (cf.\ \cite{MR1943713,MR3351079}). In \cite{MR1958838} it was shown that every representation of a Lie groupoid gives rise to a representation of $\Bis (\cG)$. The present article establishes smoothness of these representations with respect to the natural smooth structure of the bisections group. Then we prove that under mild assumptions on $\cG$ a certain class of smooth representations of the bisection group corresponds to representations of the underlying groupoid. Furthermore, we extend the link between groupoids and bisection groups to representations of another infinite-dimensional Lie group of self-maps of the Lie groupoid $\SG(\alpha)$ which we introduced in \cite{AS17}.
This yields (faithful) functors from the representation category $\Rep(\cG)$ of a Lie groupoid to the representation categories of the infinite-dimensional Lie groups. For $\alpha$-proper groupoids (defined below)  the construction commutes with restriction to the object space, whence we obtain a commuting diagram:
\begin{displaymath}
  \xymatrix{
      &  \Rep (\cG)  \ar[ld]_{\rho} \ar[rd]^{\rho_S}    &      \\
       \Rep (\Bis (\cG))         & & \Rep (\SG(\alpha)) \ar@{.>}[ll]^{\stackrel{\text{restriction}}{\text{if $\cG$ is $\alpha$-proper}}}
  }
\end{displaymath}
Our results thus link the representation theory of (finite dimensional) Lie groupoids and infinite-dimensional Lie groups and provide a foundation to study certain representations of a class of infinite-dimensional Lie groups via finite-dimensional Lie groupoids or vice versa. \medskip

We will now explain our results in greater detail. For a Lie groupoid $\cG = (G \toto M)$ we denote by $\Bis (\cG)$ its group of bisections, i.e.\ the group of smooth sections $\sigma \colon M \rightarrow G$ of the target map $\beta$ such that the composition with the source map yields a diffeomorphism $\alpha \circ \sigma$ on $M$.
It has been shown in \cite{MR3351079} that for compact source $M$, the inclusion $\Bis (\cG) \subseteq C^\infty (M,G)$ endows the bisection group with a natural manifold structure turning it into a Lie group modelled on a \Frechet\, space. To make sense of this statement (which requires calculus beyond Banach spaces), we base our investigation on smoothness in the sense of Bastiani calculus \cite{bastiani}.\footnote{A map $f \colon E\supseteq X \rightarrow F$ between (open) subsets of locally convex spaces is smooth in this sense if all iterated partial derivatives exist and yield continuous mapping, cf.\ Appendix \ref{App:Mfd}. We remark that this is a stronger notion of smoothness (for general locally convex spaces) then the so called "convenient calculus", cf.\ \cite{conv1997}.} In the present paper we do not wish to restrict ourselves to compact $M$, whence our first result of independent interest is the following.\medskip

\textbf{Proposition A} \emph{Let $\cG = (G \toto M)$ be a finite-dimensional Lie groupoid. Then $\Bis (\cG)$ is a submanifold of the manifold of mappings $C^\infty (M,G)$ (cf.\ \cite{michor1980}) and this structure turns $\Bis (\cG)$ into an infinite-dimensional Lie group.}\medskip

To our knowledge a full proof of this fact appears in full detail for the first time in the present paper (cf.\ the sketch of a generalisation in \cite{HS2016}). In view of Proposition A, our aim is to prove that a smooth representation of a Lie groupoid $\cG$ induces a smooth representation of the infinite-dimensional Lie group $\Bis (\cG)$.
To this end recall that a representation of a Lie groupoid $\cG = (G \toto M)$ on a vector bundle $E \rightarrow M$ is a morphism of Lie groupoids $\phi \colon \cG \rightarrow \Phi (E)$, where $\Phi (E)$ is the frame groupoid associated to the bundle $E$ \cite{Mackenzie05}. It has been shown in \cite{MR1958838} that every representation of $\cG$ gives rise to an associated representation $\rho_\phi \colon \Bis (\cG)\rightarrow \GL (\Gamma (E))$, where $\Gamma(E)$ is the space of smooth sections of the bundle $E\rightarrow M$. Since the natural compact open $C^\infty$-topology turns $\Gamma (E)$ into a locally convex space, it makes sense to study continuity and smoothness of $\rho_\phi$ via the associated map\footnote{Recall that the definition of smoothness of a representation $\rho \colon H \rightarrow \GL(V)$ of an infinite-dimensional Lie group (on an infinite-dimensional vector space) avoids topological data on $\GL(V)$. Instead, one requires smoothness of the orbit maps $\rho_v \colon H \rightarrow V, h \mapsto \rho(h).v$.} 
$$\rho_{\phi}^\wedge \colon \Bis (\cG) \times \Gamma (E) \rightarrow \Gamma (E),\quad (\sigma, X) \mapsto \rho_\phi (\sigma).X.$$
If the base manifold $M$ is non-compact, continuity and smoothness of $\rho_\phi^\wedge$ are difficult to acquire on $\Gamma(E)$ due to some technical obstructions enforced by the function space topologies involved. We discuss this in Remark \ref{rem:whynot?} and propose a different strategy in the present paper. The representation $\rho_\phi^\wedge$ factors through a representation on the space of compactly supported sections $\Gamma_c (E)$ of the bundle $E$. This space is again a locally convex space, though with respect to the much finer fine very strong topology (a Whitney type topology which coincides with the compact open $C^\infty$-topology if $M$ is compact, see \ref{fS:top}). By abuse of notation we also denote this representation by $\rho_\phi^\wedge$ and show that $\rho_\phi^\wedge$ is smooth in a very strong sense.
Namely, $\rho_\phi^\wedge \colon \Bis (\cG) \times \Gamma_c (E) \rightarrow \Gamma_c (E)$ turns out to be smooth and we call this property \emph{joint smoothness} of the representation. Thus every smooth representation of a Lie groupoid gives rise to a (joint) smooth representation of infinite-dimensional Lie groups.

In the representation theory of $C^*$-algebras, a partial converse for the above construction has been established in the topological category: Consider a topological groupoid which satisfies certain assumptions (such as having enough bisections, to be defined below), then a certain class of continuous  representation of the group of continuous bisections (viewed as a topological group) gives rise to a continuous representation of the topological groupoid (see \cite{MR2844451}).
 To make this statement precise, let us review the properties inherited by the representation. It turns out that $\rho_\phi$:
 \begin{itemize}
 \item is semi-linear, i.e.\ there are smooth diffeomorphisms $\mu_\sigma$ of $M$ such that
$$\rho_\phi (\sigma).(fX) = (f\circ \mu_\sigma) \cdot\big( \rho_\phi (\sigma).X \big)\quad X\in \Gamma_c(E), f \in C^\infty (M), \sigma \in \Bis (\cG).$$
\item is local, i.e.\ if a bisection maps $m$ to the unit arrow $1_m$, then $\rho_\phi (\sigma).X(m)=X(m)$ for all $X\in\Gamma_c(E)$.
\end{itemize}
 These properties already characterise the class of representations of $\Bis (\cG)$ which gives rise to representations of the underlying groupoid. Namely, we obtain our second main result.\medskip

 \textbf{Theorem B} \emph{Let $\rho \colon \cG \rightarrow \Phi(E)$ be a representation of a Lie groupoid on a vector bundle. Then $\phi$ induces a joint smooth, semi-linear and local representation $\rho_\phi$ of the bisection group $\Bis (\cG)$ on the compactly supported sections of $E$.
 Furthermore, this construction yields a faithful functor $\Rep (\cG) \rightarrow \Rep (\Bis(\cG))$ of the representation categories.}

 \emph{If in addition, $\cG$ has enough bisections\footnote{A Lie groupoid $\cG = (G\toto M)$ has "enough bisections" if for every $g \in G$ there is $\sigma_g \in \Bis (\cG)$ such that $\sigma_g (\beta(g))=g$.}, then every joint smooth, semi-linear and local representation of $\Bis (\cG)$ on a space of compactly supported sections arises in this way.}\medskip

 Theorem B allows one to characterise (under suitable assumptions) the smoothness of semi-linear representations in terms of the smoothness of $\Bis (\cG) \times \Gamma_c (E) \rightarrow \Gamma_c(E)$. In this respect, it provides an answer to the question raised in \cite[p.\ 230]{MR1958838}, whether smoothness of these representations can be characterised in terms of the action of the bisections alone.

 Furthermore, it is worth noting that we just needed to assume that the Lie groupoid has enough bisections to invert the correspondence between smooth representation. In particular, the technical assumptions needed in the topological category in \cite{MR2844451} are not needed for Lie groupoids. In ibid.\ a technical condition on the compact open topology on $C(M,G)$ had to be assumed and it was conjectured that this condition is not needed for Lie groupoids.
 Finally, we remark that the condition of having enough bisections is very natural if one wants to link Lie groupoids and their bisections. Indeed in \cite{MR3569066,MR3573833} the authors has shown that this condition is necessary to reconstruct the Lie groupoid from its group of bisections. Thus it comes as no surprise that this condition is crucial for the recovery of groupoid representations from representations of the bisection group.

In the last part of the article we consider induced representations on a Lie group of self-mappings of a Lie groupoid $\cG = (G\toto M)$. The Lie group we envisage here is the group of units $\SG(\alpha)$ of the monoid
$$(\SG \coloneq \big\{f \in C^\infty (G,G) \mid \beta \circ f = \alpha \big\}, \star),$$
where the multiplication is defined via $(f\star g) (x)= f(x)g(xf(x))$ (cf.\ \cite{AS17}). This group is closely associated to the bisection group $\Bis (\cG)$, since the latter can be realised as a subgroup of $\SG(\alpha)$.
Similar to the bisection case, representations of the Lie groupoid induce semi-linear representations of $\SG(\alpha)$.
\medskip

\textbf{Theorem C} \emph{Let $\phi \colon \cG \rightarrow \Phi (E)$ be a representation of the Lie groupoid $\cG$ on a vector bundle $(E,\pi,M)$. Then $\phi$ induces a joint smooth semi-linear and local representation $\rho_{\phi,S}$ of $\SG(\alpha)$ on $\Gamma_c (\alpha^*E)$ (compactly supported sections of the pullback bundle $\alpha^*E$). Further, the construction induces a faithful functor $\Rep (\cG) \rightarrow \Rep (\SG(\alpha))$.}\medskip

Unfortunately, there seems to be no natural condition (akin to the bisection case) which guarantees that every semi-linear, local and joint smooth representation of $\SG(\alpha)$ arises via the construction outlined in Theorem C. Nevertheless, it is interesting to study the relation between the induced representations of $\SG(\alpha)$ and $\Bis (\cG)$. To this end, recall from \cite{AS17,Amiri2017} that precomposition with the source map of $\cG$ induces a group monomorphism $\alpha_* \colon \Bis (\cG) \rightarrow \SG(\alpha)$. In general, $\alpha_*$ will not be smooth. However, if one requires that $\cG$ is $\alpha$-proper (i.e.\ $\alpha$ is a proper map), then $\alpha_*$ becomes a morphism of Lie groups. Using this morphism, we can restrict representations of $\SG(\alpha)$ to representations of $\Bis(\cG)$.
Our results subsume the following Theorem.\medskip

\textbf{Theorem D} \emph{Let $\cG$ be an $\alpha$-proper Lie groupoid, $(E,\pi,M)$ a vector bundle and $\rho_S \colon \SG(\alpha) \times \Gamma_c (\alpha^*E) \rightarrow \Gamma_c(\alpha^*E)$ a joint smooth and semi-linear representation. Then $\rho_S$ induces a joint smooth and semi-linear representation of $\Bis (\cG)$ on $\Gamma_c(E)$.}\medskip

Thus for $\alpha$-proper Lie groupoids, the constructions are compatible in the sense that they coincide for the bisection group. Since we are working with compactly supported sections, note that the restriction makes sense only for $\alpha$-proper groupoids.

The present article is structured as follows: In Section \ref{sect: prelim} we recall basic concepts and prepare the approach by establishing the Lie group structure of $\Bis (\cG)$ for non-compact source manifolds (Proposition A). Then Section \ref{sect: semi} recalls some essential results on the group of semi-linear morphisms of a $C^\infty_c$-module. Finally, Section \ref{sect:reps} deals first with representations on the group of bisections, then with the representations of $\SG(\alpha)$ and finally with functorial aspects of these constructions.

\section{Infinite-dimensional Lie groups from Lie groupoids}\label{sect: prelim}

In this section we recall first some well-known concepts and fix the notation used throughout the paper. Then we establish the Lie group structure of the bisection group for non-compact source manifolds (Proposition A from the introduction).
We assume that the reader is familiar with finite-dimensional Lie groupoids as presented e.g.\ in the book \cite{Mackenzie05}.

\begin{setup}[Conventions] We write $\N \coloneq \{1,2,3,\ldots\}$ and $\N_0 \coloneq \N \cup \{0\}$.
By $\cG = (G \toto M)$ we denote a finite-dimensional Lie groupoid with space of arrows $G$ and base $M$.\footnote{One could allow $G$ to be infinite-dimensional manifold modelled on a locally convex space. Note that the notion of an infinite-dimensional Lie groupoid makes sense once a suitable notion of submersion is chosen, see \cite{MR3351079} and \cite{Glofun}.}
In general $\alpha$ will be the source and $\beta$ the target map of $\cG$. We suppress the unit map $1\colon M\rightarrow G$ and identify $M \cong 1(M) \subseteq G$.

 Throughout the paper, all manifolds are assumed to be Hausdorff, and if a manifold is finite-dimensional we assume in addition that it is paracompact. Appendix \ref{App:Mfd} contains an outline on the infinite-dimensional calculus and facts on manifolds of mappings (such as the group of diffeomorphisms $\Diff (M)$) needed in the article.
 For $M,N$ finite-dimensional manifolds, we denote by $C^\infty_{\mathrm{fS}} (M,N)$ the space of smooth mappings $M\rightarrow N$ endowed with the fine very strong topology (a Whitney type topology turning $C^\infty (M,N)$ into a manifold which coincides with the compact open $C^\infty$-topology if $M$ is compact). See \ref{fS:top} for details.
\end{setup}


\begin{defn}[{Bisections of a Lie groupoid}]\label{defn: bis}
A \emph{bisection} of a Lie groupoid $\cG$ is a smooth map $\sigma\colon  M\rightarrow G$ which satisfies
$$ \beta \circ \sigma = \id_M \text{ and } \alpha \circ \sigma \in \Diff (M).$$
The set of all bisections $\Bis(\cG)$ is a group with respect to the group operations
\begin{equation}\label{eq: bis mult}
 (\sigma_1\star\sigma_2)(x)=\sigma_1(x)\sigma_2\big((\alpha\circ\sigma_1)(x)\big),\quad \sigma^{-1}=\Big(\sigma\circ(\alpha\circ\sigma)^{-1}(x)\Big)^{-1} \quad\text{for all}\ x \in M.
\end{equation}
Finally, we define the subgroup 
$$\Bis_c (\cG):=\{\sigma \in \Bis (\cG)\mid  \text{there is } K\subseteq M \text{ compact, with } \sigma(x)=1_x \text{ for all } x\in M\setminus K\}.$$
\end{defn}

Following \cite[Theorem A]{MR3351079} we recall that $\Bis (\cG)$ is an infinite-dimensional Lie group if $M$ is compact. For the non-compact case, we refer to \cite{HS2016} for a sketch of the changes.
However, since no complete proof exists in the literature, we restrict to the case of a finite-dimensional groupoid $\cG$ and establish the necessary details in the following Proposition.

\begin{prop}\label{prop: BIS:LGP}
Let $\cG = (G\toto M)$ be a finite-dimensional Lie groupoid. Then $\Bis_c (\cG)$ and $\Bis(\cG)$ are submanifolds of $C^\infty_{\mathrm{fS}} (M,G)$ and this manifold structure turns them into infinite-dimensional Lie groups.
\end{prop}

\begin{proof}
Recall that $\Diff (M)$ is an open submanifold of $\CfS (M,M)$ (see \cite[Theorem 10.4]{michor1980} for the manifold structure). Further, $\alpha_* \colon \CfS (M,G) \rightarrow C^\infty (M,M), f \mapsto \alpha \circ f$ is continuous (even smooth) by \cite[Corollary 10.14]{michor1980}. Since $\beta \colon G \rightarrow M$ is a submersion, the Stacey-Roberts Lemma \cite[Lemma 2.4]{AS17} asserts that $\beta_* \colon \CfS (M,G) \rightarrow \CfS (M,M)$ is a submersion, whence the restriction $\theta \coloneq \beta_*|_{\alpha_*^{-1} (\Diff (M))}$ is a submersion. We deduce that $\theta^{-1} (\id_M) = \Bis (\cG)$ is a submanifold of $C^\infty (M,G)$. Similarly, since $\Diff_c (M) = \{ g \in \Diff (M) \mid \exists K \subseteq M \text{ compact such that } g|_{M\setminus K} \equiv \id_M\}$ is an open subgroup of $\Diff (M)$, we see that $\kappa \coloneq \beta_*|_{\alpha_*^{-1} \big(\Diff_c (M)\big)}$ is a submersion and $\Bis_c (\cG) = \kappa^{-1} (\id_M)$ is thus an open submanifold of $\Bis (\cG)$ and also of $\CfS (M,G)$.

To see that these manifold structures turn $\Bis_c (\cG)$ and $\Bis (\cG)$ into Lie groups, one rewrites the formulae for the group operations as follows:
\begin{align*}
\sigma \star \tau = m_* \big(\sigma, \text{Comp} (\tau, \alpha_* (\sigma))\big) \qquad
\sigma^{-1} = I_* \circ \text{Comp} (\sigma , \iota \circ \alpha_* (\sigma))
\end{align*}
where $m$ is multiplication, $I$ inversion in $\cG$ and $\iota$ is inversion in the Lie group $\Diff (M)$ (cf.\ \cite[Theorem 11.11]{michor1980}). We have already seen that the postcomposition mappings are smooth and since $\alpha_* (\Bis (\cG)) \subseteq \Diff (M)$ and all diffeomorphisms are proper, we exploit that the composition map $\text{Comp} \colon C^\infty (M,G) \times \text{Prop} (M,M) \rightarrow C^\infty (M,G)$ is smooth by \cite[Theorem 11.4]{michor1980}. In conclusion the group operations are smooth as composition of smooth mappings
\end{proof}

In the (inequivalent) convenient setting of global analysis a similar result was established in \cite{MR1943713} for bisection groups of finite-dimensional Lie groupoids.
\begin{rem}
Similar to the arguments in \cite{MR3351079} one can show that the Lie groups $\Bis (\cG)$ and $\Bis_c (\cG)$ are regular (in the sense of Milnor) and their Lie algebras are (anti-)isomorphic to the Lie algebra of sections of the Lie algebroid associated to $\cG$.

Furthermore, we remark that the Lie group of bisections contains a surprising amount of information on the underlying Lie groupoid. For example, one can use it to reconstruct the Lie groupoid from the action of the bisections on the base manifold under certain circumstances. We refer to \cite{MR3573833} for more information.

However, neither of these results are needed in the present paper.
\end{rem}

The other infinite-dimensional Lie group we wish to consider arises as a group of units of a monoid of smooth self maps of the Lie groupoid (cf.\ \cite{Amiri2017,AS17}).
\begin{defn}\label{defn: SG}
Let $\cG = (G \toto M)$ be a Lie groupoid, then we define the set
$$\SG \coloneq \big\{f \in C^\infty (G,G) \mid \beta \circ f = \alpha \big\},$$
which becomes a monoid with respect to the binary operation
$$(f \star g) (x) = f(x)g\big(xf(x)\big), \quad f,g \in S_\cG , \ x \in G.$$
The source map $\alpha$ is the unit element of $S_\cG$ (Recall the identification $M \subseteq G$!). Finally, we let $\SG (\alpha) \coloneq \{ f \in \SG \mid f \text{ is invertible in } \SG\}$ be the group of units of $\SG$.
\end{defn}

Note that $\SG(\alpha)$ can be characterised as $\{f\in \SG \mid R(f) \in \Diff (G), \text{ where } R_f (x) \coloneq xf(x)\}.$
It was shown in \cite[Theorem B]{AS17} that $\SG (\alpha)$ is an infinite-dimensional Lie group such that the map $R\colon \SG(\alpha) \rightarrow \Diff (G), f \mapsto R_f$ is a diffeomorphism onto a closed Lie subgroup of $\Diff (G)$ (in fact $R$ is a Lie group anti-homomorphism).

\begin{rem}
 \begin{enumerate}
  \item In general, the monoid $\SG$ will not be a differentiable monoid (though its group of units will always be a Lie group!). If one assumes that the underlying Lie groupoid is $\alpha$-proper\footnote{A Lie groupoid is $\alpha$-proper if the source map $\alpha$ is a proper mapping.} then $\SG$ is a differentiable monoid.
  \item By switching the r\^{o}les of $\alpha$ and $\beta$ in the definition, One can define another group of bisections and another monoid $S_\cG'$ of smooth self maps.
  However, the resulting groups and monoids are isomorphic (cf.\ e.g.\ \cite[Proposition 1]{Amiri2017} and \cite{Mackenzie05,MR3351079}), whence it suffices to consider only $\SG (\alpha)$ and $\Bis (\cG)$.
  Note that with our choice of conventions, \cite[Theorem 4.11]{AS17} asserts that the map
  $$ \Psi \colon \Bis (\cG) \rightarrow \SG (\alpha), \quad \sigma \mapsto \sigma \circ \alpha$$
  is a group monomorphism and if $\cG$ is $\alpha$-proper, $\Psi$ is even a Lie group morphism.
 \end{enumerate}\label{rem: op:sub}
\end{rem}

\subsection*{Representations of Lie groups and groupoids}\label{sect: rep}
\addcontentsline{toc}{subsection}{Representations of Lie groups and groupoids}
In this section we recall and fix notation for representations of Lie groupoids and infinite-dimensional Lie groups.
Let us first fix some notation before we turn to representations of possibly infinite-dimensional Lie groups, cf.\ \cite{neeb2006,MR2719276}).

\begin{setup}
Let $V$ be a (possibly infinite-dimensional) locally convex vector space.
We denote by $\GL (V)$ the group (with respect to composition) of continuous linear vector space automorphisms of $V$.
Recall that if $V$ is not a Banach space, $\mathcal{L} (V,V)$ does not inherit a suitable locally convex topology from $V$, i.e.\ the correspondence
$$\{\text{topological group morphism } T \rightarrow \GL(V)\} \stackrel{1:1}{\longleftrightarrow} \{ T \times V \rightarrow V \text{ cont. representation}\}$$
breaks down.
Thus we refrain from using a topological or smooth structure on $\GL (V)$.
\end{setup}

\begin{defn}\label{defn: cont:smoo}
 Let $K$ be a (possibly infinite-dimensional) Lie group and $(\rho, V)$ be a representation of $K$ on a locally convex vector space $V$ (i.e.\ $\rho \colon K \rightarrow \GL (V)$ is a group morphism). We say that
 \begin{itemize}
  \item $(\rho,V)$ is \emph{continuous} (or \emph{joint smooth}) if the action map $K \times V \rightarrow V, (k,v) \mapsto \rho (k).v$ is continuous (or smooth),
  \item a vector $v \in V$ is a \emph{$C^k$-vector}, $k \in \N \cup \{\infty \}$ if the orbit map $\rho^v \colon K \rightarrow V, k \mapsto \rho(k).v$ is a $C^k$-map. We denote by $V^k \coloneq V^k (\pi)$ the linear subspace of $C^k$-vectors,
  The representation $(\rho,V)$ is \emph{smooth} if $V^\infty$ is dense in $V$.
 \end{itemize}
\end{defn}

The continuity and smoothness of representations is tailored to avoid topological data on $\GL(V)$.
In particular, we chose to outline the concept of "smooth representations" which is usually used in infinite-dimensional representation theory (and to distinguish it from the stronger concept of joint smoothness exhibited above).
However, in the special situation we are working in, it will turn out that the representations are joint smooth and take their values in a certain subgroup of $\GL(V)$, the group of semilinear automorphisms (cf.\ \cite{MR1958838}). We review this group in Section \ref{sect: semi} below.

\begin{defn}
Let $(\rho_1,V_1)$ and $(\rho_2,V_2)$ be representations of a (possibly infinite-dimensional) Lie group $K$ on locally convex spaces. A \emph{morphism of representations} is a continuous linear map $L \colon V_1 \rightarrow V_2$ which is $K$-equivariant, i.e.\ $\rho_2 (k).L = L\circ \rho_1 (k)\ \text{for all}\ k \in K$.
Together with morphisms of representations, the representations of a Lie group $K$ form a category, which we denote by $\Rep (K)$.
\end{defn}

We now switch from representations of (possibly infinite-dimensional) Lie groups to representations of finite-dimensional Lie groupoids. Our notation here mainly follows the book \cite{Mackenzie05}.

\begin{defn}\label{Frame} Let $(E,\pi ,M)$ be a vector bundle. Let $\Phi(E)$ denote
the set of all vector space isomorphisms $\xi: E_x\rightarrow E_y$ for $x, y\in M$. Then
$\Phi(E)$ is a Lie groupoid on $M$ with respect to the following structure:\begin{itemize}
\item for $\xi\colon E_x\rightarrow E_y$, $\alpha(\xi)$ is $x$ and $\beta(\xi)$ is $y$;
\item the object inclusion map is $x\mapsto 1_x= id_{E_x}$,
\item multiplication and inversion is given by composition and inversion of linear maps.
\end{itemize}
 We call the resulting groupoid $\Phi(E)$ the \emph{frame groupoid} (or linear frame groupoid) of $(E,\pi,M)$.
\end{defn}

\begin{rem}[{\cite[p.8 and Example 1.3.4]{Mackenzie05}}]
The frame groupoid $\Phi(E)$ of a vector bundle $E$, as
defined in the Definition \ref{Frame}, is a locally trivial Lie groupoid, i.e.\ the anchor map $(\beta, \alpha)\colon \Phi(E)\rightarrow M\times M$ is a surjective submersion.
\end{rem}

\begin{defn}
Let $\cG$ be a Lie groupoid on $M$ and $(E, \pi, M)$ be a vector bundle.
Then a \emph{representation} of $\cG$ on $(E,\pi, M)$ is a base preserving morphism $\RepG\colon \cG\rightarrow \Phi(E)$ of Lie groupoids over $M$, where $\Phi(E)$ is the frame groupoid of $(E ,\pi, M)$.
\end{defn}

\begin{rem}\label{rem: linRep}
Every representation $\RepG \colon \cG \rightarrow \Phi(E)$ induces a smooth linear action\footnote{Compare \cite[Definition 1.6.1 and Proposition 1.7.2]{Mackenzie05} where it is shown that this is even equivalent (note that loc.cit.\ claims this only for $\cG$ locally trivial, though the argument holds for arbitrary Lie groupoids.}
$$\ActG \colon G \times_\alpha E \rightarrow E,\quad (g,e) \mapsto \phi(g).e,$$
where $G \times_\alpha E :=\{(g,e) \in G\times E\mid \alpha (g) = \pi(e)\}$ is the pullback manifold.
In the following we will always denote by $\ActG$ the linear action associated to the Lie groupoid representation $\RepG$.
\end{rem}

\begin{defn}
Let $\cG = (G\toto M)$ be a Lie groupoid and $(E,\pi,M), (F,q,M)$ vector bundles.
A \emph{morphism of Lie groupoid representations} $\RepG_1 \colon \cG \rightarrow \Phi(E)$ and $\RepG_2 \colon \cG \rightarrow \Phi(F)$ is a base preserving bundle morphism $\delta \colon E \rightarrow F$ which is equivariant with respect to the linear action $\ActG_1$ and $\ActG_2$, i.e.\
$$\delta \circ \ActG_1 (g,\cdot) = \ActG_2 (g,\cdot)\circ \delta \quad \forall g\in G.$$
The representations of $\cG$ form together with morphisms of representations a category which we denote by $\Rep (\cG)$
\end{defn}

\begin{ex}\label{ex: gpd:rep}
Let $(E,\pi,M)$ be a vector bundle and $K$ a Lie group.
\begin{itemize}
\item[(a)] If $M = \{ *\}$ is the one-point manifold and $G$ a Lie group, then a representation of the Lie groupoid $G \toto \{*\}$ on $E$ is a representation of the Lie group $G$ on the vector space $E$. Hence $\Rep (G) = \Rep (G\toto \{*\})$ and therefore there is no ambiguity in our notation for the categories of representations.
\item[(b)] For the pair groupoid $\mathcal{P} (M)$, a representation on $E$ corresponds to a trivialisation of the bundle $(E,\pi,M)$ (cf.\ \cite[Example 2.2]{MR3696590}).
\item[(c)] For the trivial Lie group bundle $LB(K) \coloneq (K\times M \toto M)$ over $M$ with fibre $K$ (viewed as a Lie groupoid), representations on $E=TM$ (the tangent bundle of $M$) are trivial (cf.\ \cite[Example 7.2]{MR3696590}).
\item[(d)] Let $K \times M \rightarrow M$ be a (left) Lie group action. Representations of the associated action groupoid $K \ltimes M \toto M$ correspond to $K$-equivariant bundles over $M$.
Generalising this example, the representations of an atlas groupoid $\Gamma (\mathcal{U})$ representing an orbifold $(Q,\mathcal{U})$ correspond to orbifold vector bundles over $(Q,\mathcal{U})$ which arise e.g.\ in orbifold cohomology (cf.\ \cite[Section 2.3]{MR2359514} and see \cite{AS17} for a discussion of source-properness for atlas groupoids).
\end{itemize}
Contrary to the situation for Lie groups, general Lie groupoids do not admit an "adjoint representation" (see however \ref{setup: jetgpd} and Example \ref{ex: repjt}). As a consequence, one was led to the generalised notion of a "representations up to homotopy", see \cite{MR3696590}. In the present article, we will only consider the classical concept, which appears naturally in applications (e.g.\ in the representation theory of $C^*$-algebras) and translates to the right concept of Lie group  representation as explained below.
\end{ex}

We will see in Section \ref{sect:reps} that representations of a Lie groupoid on a vector bundle induce representations of the infinite-dimensional groups on the sections of this bundle. Since smooth functions act in a natural way on sections of a vector bundle, one wonders in which way the induced representations are compatible with this structure. To answer this question we recall the notion of \emph{semi-linear} automorphisms of a module.

\section{The group of semi-linear automorphisms}\label{sect: semi}

In this section we discuss the group of semi-linear automorphisms of a module of smooth functions. It turns out that semi-linear automorphisms of a space of smooth sections are closely connected to the bisections of a the frame groupoid.

\begin{setup}
For a smooth finite dimensional manifold $M$ we consider
$$C^\infty_c (M) \coloneq \{f\in \CfS (M,\R) \mid \text{there is } K \subseteq M \text{ compact, with }  f|_{M\setminus K} \equiv 0 \}$$
with the canonical topological algebra structure.
 In general $C^\infty_c (M)$ is a non-unital topological algebra, as the constant $1$-function is not contained in $C^\infty_c(M)$ if $M$ is non-compact.
 However, if $M$ is compact, $C^\infty (M)$ is a (unital) \Frechet -algebra and even a continuous inverse algebra \cite{MR2501579}.
\end{setup}

\begin{defn}\label{semi}
Let $\mathbb{M}$ be a $C^\infty_c (M)$-module. An $\R$-linear automorphism $\mu \colon \mathbb{M} \rightarrow \mathbb{M}$ is called \emph{semi-linear isomorphism} of $\mathbb{M}$ if there exists an algebra isomorphism, $\mu^M \colon C^\infty_c (M) \rightarrow C^\infty_c (M)$ satisfying
$$\mu (f.X)=\mu^M(f).\mu(X)\ \ \text{for all } f\in C^{\infty}(M) \text{ and all}\  X\in \mathbb{M}.$$
Let $\SL (\mathbb{M})$ be the \emph{set of all semi-linear automorphisms} of $\mathbb{M}$ and note that composition turns $\SL (\mathbb{M})$ into a group.
\end{defn}

\begin{rem}\label{rem: assoc:diffeo}
Recall from \cite{MR2159792,MR0177002} that every algebra automorphism $\mu^M$ of  $C^\infty_c (M)$ is induced by a diffeomorphism $\mu_M \in \Diff (M)$, i.e.\
$$\mu^M (f) = f \circ (\mu_M)^{-1} \equalscolon \mu_M.f,\quad \forall f\in C^\infty_c (M).$$
We call $\mu_M$ the \emph{diffeomorphism associated to $\mu^M$} and denote its (right) action on $C^\infty_c (M)$ by $\mu_M.f$.
Note that a similar construction even hold for $C^\infty (M)$ where $M$ is allowed to be infinite-dimensional (see \cite{MR2161810}).
\end{rem}

\begin{defn}\label{semirep}
 Let $G$ be a Lie group and $\M$ be a $C^\infty_c (M)$-module which is in addition a locally convex vector space. We call a representation $(\rho, \M)$ of $G$ a \emph{semi-linear representation} if $\rho (G) \subseteq \SL(\M)$.
 A semi-linear representation is continuous (smooth) if it is continuous (smooth) as a representation in the sense of Definition \ref{defn: cont:smoo}.
\end{defn}

In the present article, we are primarily interested in semi-linear representations of Lie groups on the module of compactly supported sections of vector bundles over a fixed base $M$. The reason is that for non compact $M$, the topology of the spaces $\Gamma_c(E)$ is amenable to our methods, whereas the one on $\Gamma(E)$ is not. However, in the algebraic setting, e.g.\ \cite{MR1958838} and \cite{MR2501579} (where $M$ is assumed to be compact), proofs for the spaces of compactly supported and general sections are completely analogous. Thus in the (purely) algebraic setting we present the proofs for both $\Gamma (E)$ and $\Gamma_c(E)$.

\begin{ex}\label{ex: iso:sl}
Let  $(E, \pi, M)$ be a vector bundle over $M$. It is well known that $\Gamma(E)$ and $\Gamma_c (E)$ are $C^\infty_c (M)$ modules with respect to the pointwise multiplication.
As a consequence of \cite{MR1958838}, the group $\SL (\Gamma (E))$ is isomorphic to $\Bis (\Phi(E))$ and to the group of smooth bundle automorphisms $\Aut(E)$.
For the readers convenience, Proposition \ref{prop: iso:VB} recalls these details and establish an isomorphism $\SL(\Gamma_c(E)) \cong \Bis (\Phi(E))$.\footnote{The construction in \cite{MR1958838} uses the alternative convention for bisections (reversing the r\^{o}le of $\alpha$ and $\beta$). Furthermore, a crucial step uses \cite[Proposition 4]{MR0438405} for which a complete proof (for which we thank K.--H.\ Neeb) is only contained in \cite{MR2501579}.}
\end{ex}

The isomorphism in Example \ref{ex: iso:sl} allows us to turn $\SL (\Gamma_c (E))$ into an infinite-dimensional Lie group (see \cite{MR2501579} for the special case of $M$ being compact).
The following constructions are purely algebraic, whence independent of topologies on $C^\infty_c (M), \Gamma(E)$ and $\Gamma_c(E)$. As a first step we extract from \cite[Example 2.5]{MR2501579} the following.

\begin{lem}\label{lem: bunaut:SL}
Let $(E,\pi, M)$ be a smooth vector bundle and $\M \in \{ \Gamma (E), \Gamma_c (E)\}$. Then $\mu \in \SL (\M)$ induces a (smooth) bundle automorphism $\phi_\mu$ which acts by the diffeomorphism $\mu_M$ (Remark \ref{rem: assoc:diffeo}) on $M$.
\end{lem}

\begin{proof}
Consider the maximal ideal $I_m \coloneq \{f\in C^\infty_c (M) \mid f(m) = 0\}$. Using local triviality of the bundle $E$, it is easy to see \cite[Lemma 11.8]{MR1930277}, that the kernel of the evaluation map $\ev_m (X) = X(m)$ is $I_m \M$. This yields a vector space isomorphism $\M / I_m \M \cong E_m$ and we identify $E$ with the disjoint union $E = \bigsqcup_m \M /I_m\M $.

For $\mu \in \SL(\M)$ we denote the corresponding algebra isomorphism of $C^\infty_c (M)$ by $\mu^M$ and the induced diffeomorphism by $\mu_M \in \Diff (M)$ (cf.\ Remark \ref{rem: assoc:diffeo}).
Then $\mu^M (I_m) = I_{\mu_M (m)}$ and thus $\mu(X+I_{m}\M)=\mu(X)+I_{\mu_M (m)}\M$. In conclusion, the value $\mu(X)(\mu_M(m))$ does only depend on the value of $X(m)$, whence we obtain a well-defined map
\begin{equation}\label{eq: bun}
\phi_\mu \colon E \rightarrow E,\quad v \mapsto \mu (X_v)(\mu_M (\pi(v))),\quad \text{where } X_v (\pi(v))=v .
\end{equation}
Since $\mu$ is $\R$-linear, $\phi_\mu$ is linear in each fibre and it is a bijection since $\mu$ is an isomorphism of $\M$.
To see that $\phi_\mu$ is smooth, one works in a neighbourhood of $m\in M$ and chooses a subset $\{X_1,\ldots , X_k\} \subseteq \M$ such that in an $m$-neighbourhood the elements $\{X_1 (n),\ldots X_k(n)\}$ form a base of $T_nM$. Applying \eqref{eq: bun} to this family we obtain $\phi_\mu (e, \sum_{i=1}^k s_i X_i (e)) = (\mu_M (e), \sum_{i=1}^k s_i \mu (X_i)(\mu_M (e))$ which is clearly smooth.
\end{proof}

\begin{prop}\label{prop: iso:VB}
Let $(E,\pi,M)$ be a vector bundle over the finite-dimensional manifold $M$ and $\Phi (E)$ be its frame groupoid. Then there is a group isomorphism
\begin{align}
\gamma \colon \Bis (\Phi (E)) \rightarrow \SL(\Gamma_c (E)),\quad \gamma(\sigma): \xi \mapsto (x\mapsto \sigma (x).\xi (\alpha \circ \sigma (x))),\label{BIS:SL}
\end{align}
which factors through a group isomorphism $\Bis(\Phi(E)) \cong \Aut(E)$.
\end{prop}

\begin{proof} We establish first the isomorphism $\Bis(\Phi(E)) \cong \text{Aut} (E)$. Then we prove that $\text{Aut} (E)$ is isomorphic to $\SL (\Gamma_c(E))$.

\paragraph{Step 1:} \emph{$\Bis (\Phi (E))$ is isomorphic to  $\Aut (E)$.}
Consider the mapping
\begin{align}
B \colon \Bis (\Phi (E)) \rightarrow \Aut (E) , \text{ with } B(\sigma) (v) = \sigma \big((\alpha \circ \sigma)^{-1} (\pi (v))\big).v . \label{bis:bun}
\end{align}
By definition of the frame groupoid, $B(\sigma)$ is contained in $\Aut(E)$ for every $\sigma \in \Bis (\Phi(E))$. Further, it is easy to check that the automorphism $B(\sigma)$ covers $(\alpha \circ \sigma)^{-1}$ (i.e.\ $\pi \circ B(\sigma) = (\alpha \circ \sigma)^{-1} \circ \pi$).
Let us now present a bundle automorphism $F$ as a pair $(F, b_F)$, where $F$ covers $b_F \in \Diff (M)$. Then an inverse of \eqref{bis:bun} is given by the map
\begin{align}
\Aut (E) \rightarrow \Bis (\Phi (E)), (F, b_F) \mapsto (x \mapsto \sigma_F (x) \coloneq F|_{E_{b_F^{-1} (x)}})\label{bun:bis}
\end{align}
We note that \eqref{bun:bis} makes sense as smoothness of $(x \mapsto \sigma_F (x))$ can be established by local triviality of $E$ (cf.\ \cite[Proof of Proposition 5.4]{MR1958838})
An easy but tedious computation shows that $B$ and its inverse are group (iso-)morphisms, whence we can identify $\Bis (\Phi(E))$ with $\Aut (E)$.
\paragraph{Step 2:} \emph{A group monomorphism $\Aut (E) \rightarrow \SL (\Gamma_c(E))$.}
Consider the map
\begin{align}
\nu \colon \Aut (E) \rightarrow \SL (\Gamma_c (E)), (F,b_F) \mapsto (\xi \mapsto F\circ \xi \circ b_F^{-1}).\label{AUT:SL}
\end{align}
Note that for every $\xi \in \Gamma_c (E))$, also the section $\nu (F,b_F).\xi$ has compact support (as  $b_F \in \Diff (M)$). Thus \eqref{AUT:SL} makes sense and $F\circ (f\cdot \xi) \circ b_F^{-1} = (b_F^{-1})^*(f)\cdot (F\circ \xi \circ b_F^{-1})$ shows that we obtain a semi-linear isomorphism of $\Gamma_c(E)$.
A straight forward calculation shows that \eqref{AUT:SL} is a group (mono-)morphism.

\paragraph{Step 3:} \emph{The inverse of $\nu$.}
Let us explicitly construct an inverse for $\nu$. Recall from Lemma \ref{lem: bunaut:SL} that to every $\mu \in \SL (\Gamma_c (E))$ there is an associated bundle automorphism $\phi_\mu$ of $E$ Hence we can define
$$\nu^{-1} \colon \SL \big(\Gamma_c (E)\big) \rightarrow \Aut (E), \mu \mapsto (\phi_\mu, \mu_M).$$
To see that this map inverts $\nu$ we compute for $\mu \in \SL (\Gamma_c (E))$ as follows $\nu(\nu^{-1} (\mu)) = (\xi \mapsto \phi_\mu \circ \xi \circ \mu_M^{-1})$.
Evaluating the sections in $x\in M$ we see that $\nu (\nu^{-1}(\mu))(\xi)(x) = \mu(\xi)(x), \forall x \in M$, whence $\nu \circ \nu^{-1}$ is the identity on $\SL (\Gamma_c(E))$.
Conversely, let us prove that $\nu^{-1}\circ\nu$ is the identity.
If $(F,b_F)\in \Aut (E)$, then by the Step 2 $\mu(F) \coloneq \nu(F,b_F)$ satisfies $\mu (F) (\xi) = F\circ\xi\circ b_F^{-1}$, whence $\mu_F$ is a  semi-linear isomorphism on $\SL (\Gamma_c (E))$ with associated diffeomorphism $\mu(F)_M = b_F$ (cf.\ Remark \ref{rem: assoc:diffeo}).
Now we will show that $\nu^{-1}(\nu(F,b_F))=\nu^{-1}(\mu(F))=(F,b_F)$:
It is enough to show that on each fibre $E_m, m\in M$ we have equality of $\nu^{-1}(\mu(F))$ and $F$.
Consider $x\in E_m$ and $X\in\Gamma_c(E)$ be such that $\ev_m(X)=X(m)=x$.
Now
\begin{equation}\label{eq: bf:ident}
\nu^{-1}(\mu(F))(x)=\mu(F)(X)(b_F(m))=F\circ X\circ b_F^{-1}(b_F(m)) =F(x)
\end{equation}
As $x$ was arbitrary, \eqref{eq: bf:ident} implies $\nu^{-1}\circ \nu((F,b_F))=(F,b_F)$.
\smallskip

Composing the group isomorphisms $\nu$ and $B$ we obtain \eqref{BIS:SL}.
\end{proof}

We will now always endow $\SL(\Gamma_c(E))$ with the unique infinite-dimensional Lie group structure making it isomorphic to the Lie group $\Bis (\Phi(E))$ (see Proposition \ref{prop: BIS:LGP}). For $M$ compact this has been studied in \cite{MR2501579} in the context of Lie group extensions associated to actions on continuous inverse algebras. Furthermore, it is a classical result that the group $\Aut (E)$ of bundle automorphisms can be turned into an infinite-dimensional Lie group (see e.g.\ \cite{MR1040392}).
Comparing the two Lie group structures, it is then easy to see that the group isomorphism $\Bis (\Phi (E)) \cong \Aut (E)$ becomes a diffeomorphism, whence the smooth structures on both groups are isomorphic:

\begin{cor}
Let $(E,\pi, M)$ be a smooth vector bundle. Then $\Bis (\Phi (E))$ and $\Aut (E)$ are isomorphic as Lie groups.
\end{cor}

\begin{rem}\label{rem: algsame}
 Note that the proof of Proposition \ref{prop: iso:VB} did neither use the topological structure of $\Gamma_c(E)$ nor any topological information on the algebra $C^\infty_c(M)$. Hence we could without any changes replace $\Gamma_c(E)$ with the module $\Gamma (E)$. In particular, this implies $\SL(\Gamma_c(E))\cong \Aut (E) \cong \SL(\Gamma(E))$. 
\end{rem}

From the purely algebraic point of view (disregarding continuity and smoothness) Remark \ref{rem: algsame} shows that it is irrelevant whether we consider the semi-linear representation introduced in the next section on $\Gamma (E)$ or $\Gamma_c(E)$.

\section{Linking representations of Lie groupoids and infinite-dimensional Lie groups}\label{sect:reps}
In this section we will show that for a Lie groupoid $\cG$, a representation $\cG \rightarrow \Phi(E)$, induces semi-linear representation of the Lie groups $\Bis (\cG)$ and $\SG(\alpha)$. Conversely, we will prove that under suitable assumptions on $\cG$, certain semi-linear representations of $\Bis (\cG)$ and $\SG(\alpha)$ are induced by representations of the underlying Lie groupoid.

\subsection*{Representations of the bisection group}
\addcontentsline{toc}{subsection}{Representations of the bisection group}

In this section we will discuss the relation between representations of the Lie groupoid $\cG$ and of the Lie group $\Bis (\cG)$.
Our approach is inspired by the approach in \cite[Section 5]{MR1958838} and \cite[Section 3.5]{MR2844451}. Note however, that in the first case, the bisections were considered without topological or smooth structure and in the second source these questions were treated only in a topological framework (but for fields of Hilbert spaces).\footnote{In fact \cite{MR2844451} deals with topological groupoids and representations on continuous fields of Hilbert spaces. In addition $\Bis (\cG)$ (understood as the group of continuous sections) is then treated as a topological group.}

\begin{prop}\label{prop: construct1}
Let $\cG$ be a locally trivial Lie groupoid on $M$, $(E, \pi, M)$ be a vector bundle and $\RepG \colon \cG \rightarrow \Phi (E)$ be a representation of $\cG$. Then $\RepG$ induces semi-linear representation $\rhoG \colon \Bis (\cG) \rightarrow \SL\big(\Gamma(E)\big)$ of $\Bis (\cG)$ given by the formula
\begin{equation}\label{indrep: BIS}
\rhoG (\sigma) (\xi)(m) \coloneq \RepG (\sigma (m))\Big( \xi(\alpha \circ \sigma (m))\Big) , m \in M, \xi\in \Gamma_c(E).
\end{equation}
Moreover, the representation restricts to a semi-linear representation on $\Gamma_c(E)$ which we also denote by $\rhoG \colon \Bis (\cG) \rightarrow \SL\big(\Gamma_c(E)\big)$
\end{prop}

\begin{proof}
Let us first prove that $\rhoG (\sigma) \colon \Gamma(E) \rightarrow \Gamma(E)$ (or $\rhoG (\sigma) \colon \Gamma_c(E) \rightarrow \Gamma_c(E)$) makes sense.

\paragraph{Step 1:} \emph{$\rhoG(\sigma)(\xi)$ is a (compactly supported) section if $\xi$ is a (compactly supported) section.} Note that by \eqref{indrep: BIS}, we have
 $$\pi \Big((\rhoG (\sigma) (\xi))(m)\Big) = \pi\Big( (\RepG(\sigma (m))(\xi(\alpha \circ \sigma (m))))\Big) = \beta (\sigma (m)) = m,$$
 whence the image is again a section.

Let now $\xi\in\Gamma_c(E)$, $K=\text{supp} \xi$ and $\sigma\in\Bis(\cG)$. Then $\rhoG (\sigma)(\xi)(m) = \RepG(\sigma (m))\Big( \xi(\alpha \circ \sigma (m))\Big)$ for every $m\in M$. Since $\RepG(\sigma(m))$ is an isomorphism, $\rhoG (\sigma)(\xi)(m)\notin \im 0_M$ if and only if $\xi(\alpha \circ \sigma (m))\neq 0$. Therefore $\alpha \circ \sigma (m)\in K$ and $m\in(\alpha\circ\sigma)^{-1}(K)$.
Hence
$$\{m\in M\colon \RepG(\sigma)(\xi)(m)\notin \im 0_M\}\subseteq (\alpha\circ\sigma)^{-1}(K),$$ which is compact since $\alpha\circ\sigma\in \Diff(M)$. Consequently $\rhoG (\sigma)\xi$ is compactly supported.

\paragraph{Step 2:} \emph{$\rhoG (\sigma) (\xi)$ is smooth, whence $\rhoG(\sigma) \colon \Gamma(E) \rightarrow \Gamma(E)$ makes sense.}
 Since $\RepG$ is a representation of the Lie groupoid $\cG$ on $E$, \cite[Proposition 1.7.2]{Mackenzie05} implies that $\ActG \colon G \times_\alpha E\rightarrow E, (g, x)\mapsto \RepG (g)(x)$ is smooth, where the pullback manifold is given as $G\times_\alpha E\coloneq \{(g,x) \colon x\in E_{\alpha(g)}\}$.
 We conclude that $m\mapsto \RepG (\sigma (m))\Big( \xi(\alpha \circ \sigma (m))\Big)$ is smooth as a composition of the smooth maps $M\rightarrow G\times_\alpha E, m\mapsto\Big(\sigma(m),\xi\big(\alpha\circ\sigma(m)\big)\Big)$ and $\ActG \colon G\times_\alpha E \rightarrow E$ and therefore it is smooth and $\rhoG  (\sigma) (\xi)\in\Gamma(E)$ or $\rhoG  (\sigma) (\xi)\in\Gamma_c(E)$ if $\xi \in \Gamma_c(E)$.

 \paragraph{Step 3:} \emph{$\rhoG$ is a semi-linear representation of $\Bis (\cG)$ on $\Gamma (E)$ whence also on $\Gamma_c(E)$.}
 It is straightforward to check that $\rhoG (\sigma_1\star\sigma_2)=\rhoG (\sigma_1)\circ \rhoG (\sigma_2)$ for $\sigma_1, \sigma_2\in \Bis(\cG)$. As $\rhoG  (\sigma)$ is clearly linear, $\rhoG $ defines a group morphism from $\Bis (\cG)$ to $\GL (\Gamma (E))$ and thus $\rhoG$ is a representation of $\Bis (\cG)$ on $\Gamma (E)$.
 A direct computation shows that $\rhoG(\sigma)(fX) = f\circ (\alpha \circ \sigma) . \rhoG(\sigma)(X)$ holds for all $f\in C^\infty(M)$ and $X \in \Gamma(E)$. Thus $\rhoG(\sigma) \in \SL(\Gamma(E)) \cong \SL (\Gamma_c(E))$ with associated diffeomorphism $\alpha \circ \sigma$.
\end{proof}

The representations $\rho_\phi$ we have constructed for $\Bis (\cG)$ are local in the following sense:

\begin{rem}\label{locality}
Let $\phi \colon \cG \rightarrow \Phi (E)$ be a representation of a Lie groupoid $\cG = (G\toto M)$, $\M \in \{\Gamma(E),\Gamma_c(E)\}$ and $\rho_\phi \colon \Bis (\cG) \rightarrow \SL(\M)$ be the semi-linear representation from Proposition \ref{prop: construct1}.
A trivial computation using \eqref{indrep: BIS} shows that $\rho_\phi$ satisfies the following \emph{locallity condition}:
$$\sigma\in \Bis(\cG), \sigma(m)=m \in M\quad \text{then}\quad \big(\rhoG (\sigma)\xi)(m)=\xi(m)\quad\text{for all}\quad\xi\in\M.$$
\end{rem}

Locality can conveniently be phrased in terms of the natural $\SG(\alpha)$-action on $G$.

\begin{setup}[Natural $\SG (\alpha)$ action]\label{setup:sgalph:act}
For $\cG = (G\toto M)$, the group $\SG (\alpha)$ acts on $G$ via
\begin{align*}
\gamma_{\SG} \colon \SG (\alpha) \times G &\rightarrow G, (f,x) \mapsto f.x \coloneq x\cdot f(x).
\end{align*}
Using the group morphism $\Psi$ (Remark \ref{rem: op:sub}), we obtain an action\footnote{Note that this action is not the natural action $\Bis (\cG)$, i.e.\ $\Bis (\cG) \times M \rightarrow M , (\sigma,m) \mapsto \alpha \circ \sigma (m)$} of $\Bis (\cG)$ on $G$:
\begin{displaymath}
\gamma_{\Bis} \colon \Bis (\cG) \times G \rightarrow G, (\sigma, x) \mapsto \Psi (\sigma).x= x\cdot \sigma(\alpha (x))
\end{displaymath}
Finally, we remark that $\gamma_{\SG}$ and $\gamma_{\Bis}$ are smooth, whence Lie group actions.
\end{setup}

Now the $\SG(\alpha)$-action allows us to rephrase the concept of a local action discussed in Remark \ref{locality}:

\begin{defn}\label{locality1}
Let $\cG = (G\toto M)$ be a Lie groupoid and $(E,\pi,M)$ be vector bundle and $\M \in \{\Gamma(E),\Gamma_c(E)\}$.
We say that a representation $\rho \colon \Bis (\cG) \rightarrow \GL (\M)$ is \emph{local} if
$$\gamma_{\Bis} (\sigma , m) = m \text{ implies }\rho(\sigma)\xi (m) = \xi (m) \text{ for all }\xi \in \M.$$
Similarly, for $\mathbb{L}\in \{\Gamma(\alpha^*E),\Gamma_c(\alpha^*E)\}$ a representation $\hat{\rho} \colon \SG (\alpha) \rightarrow \GL (\mathbb{L})$ is local if
$$\gamma_{\SG} (f , g) = g \text{ implies }\hat{\rho}(f)\xi (g) = \xi (g) \text{ for all }\xi \in \mathbb{L}.$$
\end{defn}

\begin{cor}
Let $\cG$ be a Lie groupoid on $M$, $(E, \pi, M)$ be a vector bundle and $\RepG \colon \cG \rightarrow \Phi (E)$ be a representation of $\cG$. Then the induced semi-linear representation $\rho_\phi$ is local.
\end{cor}

Having established the algebraic correspondence between representations of a Lie groupoid and its bisection group we turn now to the smoothness of the induced representation.

\begin{prop}\label{prop:smooth}
Let $\RepG$ be a representation of the Lie groupoid $\cG$ on the vector bundle $E$, then the map $\rhoG \colon \Bis (\cG) \rightarrow \SL (\Gamma_c (E))$ from Proposition \ref{prop: construct1} is joint smooth, i.e.\ the associated map $\rho_\phi^\wedge \colon \Bis (\cG) \times \Gamma_c(E)\rightarrow \Gamma_c (E), (\sigma,X) \mapsto \rho_\phi (\sigma).X$ is smooth.
\end{prop}
\begin{proof}
Let us first rewrite the formula for the representation with the help of the linear action map (cf.\ Remark \ref{rem: linRep}):
\begin{align}
\rhoG (\sigma)(\xi) = \ActG \circ (\sigma , \xi \circ \alpha \circ \sigma) = \ActG_* \circ (\sigma , \text{Comp} (\xi , \alpha_* (\sigma))).\label{formula1}\end{align}
Here we used $\ActG_* \colon C^\infty (M,G\times_\alpha E) \rightarrow C^\infty (M,E), f \mapsto \ActG \circ f$, $\alpha_* \colon \Bis (\cG) \rightarrow C^\infty (M,M)$ and $\text{Comp} \colon C^\infty (M,E) \times \text{Prop} (M,M), (g,h) \mapsto g\circ h$ where $\text{Prop}(M,M)$ denotes the set of proper mappings in $C^\infty (M,M)$. Note that the signature of $\text{Comp}$ makes sense as $\alpha_* (\Bis (\cG)) \subseteq \Diff (M) \subseteq \text{Prop} (M,M)$.

We claim that all the mappings in \eqref{formula1} are smooth as mappings of $(\sigma,\xi)$ and thus $\rhoG$ is smooth.
To see this recall that for a smooth mapping $f \colon K \rightarrow L$ between manifolds, the map $f_* \colon \CfS (M,K) \rightarrow \CfS (M,L)$ is smooth \cite[Corollary 10.14]{michor1980}.
Now $G\times_\alpha E \subseteq G \times E$ is a split submanifold, whence $C^\infty (M, G\times_\alpha E) \subseteq C^\infty (M, G\times E)$ is a closed submanifold (combine \cite[Lemma 1.13]{Glofun} and \cite[Proposition 10.8]{michor1980}). In addition $\Bis (\cG)$ is a submanifold of $\CfS (M,G)$ and we can thus conclude that $\alpha_*$ and $\ActG_*$ in \eqref{formula1} are smooth (using again \cite[Proposition 1.7.2]{Mackenzie05} for the smoothness of $\ActG$). Further $\text{Comp}$ is smooth by \cite[Theorem 11.4]{michor1980}. This already proves that $\rhoG$ is smooth as a mapping into $\CfS (M, E)$

To conclude the proof we note that $\Gamma_c (E)$ is a locally convex space and a closed submanifold of $\CfS (M,E)$ (it is the preimage $\pi_*^{-1}(\id_M)$ of the smooth mapping $\pi_*$, compare \cite[Corollary 10.14]{michor1980}). As $\rhoG$ takes its image in the closed submanifold, $\rhoG$ is indeed smooth as a mapping into $\Gamma_c (E)$.
\end{proof}

Proposition \ref{prop:smooth} has only be established for compactly supported sections and we will now point out the reasons for not establishing it for the space of all sections.

\begin{rem}\label{rem:whynot?}
To establish continuity and joint smoothness of $\rho_\phi \colon \Bis (\cG) \rightarrow \SL (\Gamma_c (E))$, we exploited that $\text{Comp}\colon \Gamma_c (E) \times \text{Prop} (M,M) \rightarrow \CfS (M,E)$ is smooth. This does not carry over to the $\Gamma (E)$ case as $\Gamma(E) \subseteq C^\infty (M,E)$ carries the compact open $C^\infty$-topology which does not turn $C^\infty (M,E)$ into a manifold (cf.\ \cite[p.\ 429]{conv1997}). Hence it makes no sense to discuss smoothness of the analogous joint composition map. 
Though this does not rule out that the induced representations could be (joint) smooth on $\Gamma(E)$, the proof strategy of Proposition \ref{prop:smooth} fails. As smoothness of (joint) composition maps is a powerful tool (whose proof is quite involved, cf.\ \cite[Theorem 11.4]{michor1980}) we expect that an analogue of Proposition \ref{prop:smooth} for $\rho_\phi \colon \Bis (\cG) \rightarrow \SL (\Gamma(E))$ would require a serious technical effort using specialised arguments. 

We wish to stress that this problem arises only if the source manifold $M$ is non-compact. For compact $M$ Proposition \ref{prop:smooth} trivially holds for $\Gamma(E)$ as it coincides with $\Gamma_c (E)$. Moreover, it stands to reason that for non-compact $M$ the spaces $\Gamma_c (E)$ with the fine very strong topology might actually be the more relevant space in infinite-dimensional geometry as they are the model spaces of manifolds of mappings.    
\end{rem}

As we have only established (joint) smoothness of the representations $\rho_\phi$ on spaces of compactly supported sections we will from now on restrict our attention to this case. The idea is to construct from  a smooth representation on compactly supported sections a (smooth) representation of the underlying Lie groupoid. 
To this end, we need to consider Lie groupoids for which the groups $\Bis (\cG)$ and $\SG (\alpha)$ encode much of the relevant information on the underlying groupoid (see e.g.\ \cite{MR3573833,MR3569066} and \cite{AS17} for further information on this topic). 

\begin{defn}
A Lie groupoid $\cG = (G \toto M)$ is said to
 have \emph{enough bisections}, if for every $g \in G$ there is $\sigma_g \in \Bis (\cG)$ with $\sigma_g (\beta (g)) = g$.
\end{defn}

We will now establish a converse to Proposition \ref{prop:smooth} and construct representations of the underlying Lie groupoid if it has enough bisections.

\begin{prop}\label{prop:construction3}
Let $\cG =(G\toto M)$ be a Lie groupoid which has enough bisection, $(E, \pi, M)$ is a smooth vector bundle and $\rho \colon \Bis(\cG) \rightarrow \SL\big(\Gamma_c(E)\big)$ is a joint smooth local representation, then $\rho=\rhoG $ for a smooth representation $\phi \colon \cG\rightarrow\Phi(E)$ (cf.\ Proposition \ref{prop: construct1}).
\end{prop}

\begin{proof} For convenience set $\M=\Gamma_c(E)$ and consider $\sigma\in \Bis(\cG)$. Then $\rho(\sigma)\colon \M\rightarrow \M$ is a semi-linear isomorphism and due to Proposition \ref{prop: iso:VB} we have the following data associated to $\rho(\sigma)$:
\begin{itemize}
 \item a diffeomorphism $\rho(\sigma)_M$ (cf.\ Remark \ref{rem: assoc:diffeo}), induced by a bisection of the frame groupoid, i.e.\ for every $g\in G$ we have $\rho(\sigma)_M.(\alpha(g))=\beta(g)$ (due to Step 1 of  Proposition \ref{prop: iso:VB}),
\item a bundle automorphism $F_{\rho(\sigma)}\colon E \rightarrow E$ which covers $\rho(\sigma)_M$.
\end{itemize}
We can thus define a representation of the groupoid $\cG$ on $E$ as follows.
Since $\cG$ has enough bisection, for each $g\in G$ choose a bisection $\sigma_g\in \Bis(\cG)$ with $\sigma_g\big(\beta(g)\big)=g$. Also \cite[Lemma 11.8]{MR1930277} implies that for $h\in E$ with $h\in\pi^{-1}(\alpha(g))$ there exists a section $\xi_h\in \M$ with $\xi_h(\alpha(g))=h$. We now define the action
\begin{equation}\label{indrep gpd:Bis}
\phi\colon G\rightarrow \Phi(E), \quad \phi(g)h\coloneq \sigma_g.h\coloneq (\rho(\sigma_g).\xi_h)\big(\rho (\sigma_g)_M.(\alpha(g))\big) = F_{\rho(\sigma_g)}(h)
\end{equation}
and note that by  definition $\phi$ does not depend on the choice of $\xi_h$ (by the computation in Step 3 of Proposition \ref{prop: iso:VB}).

We have to prove that the definition of $\phi(g)$ is independent of the choice of $\sigma_g\in \Bis(\cG)$.  To this end consider, $\sigma_g, \sigma_g'\in \Bis(\cG)$ with $\sigma_g(\beta(g))=\sigma_g'(\beta(g))=g$,  whence $(\sigma_g\star\sigma_g'^{-1})(\beta(g))=\beta(g)$ (remember the identification $M\subseteq G$).
Since $\rho$ is a local representation (see the Remark \ref{locality}), we conclude that for every  $\xi\in\M$
$$\Big(\rho\big(\sigma_g\star(\sigma_g'^{-1})\big)\xi-\xi\Big)(\beta(g))=0,$$ Replacing $\xi$ by
$\eta \coloneq (\rho(\sigma_g')\xi)$ we obtain $(\rho(\sigma_g)\xi)(\beta(g))=(\rho(\sigma_g')\xi)(\beta(g))$, whence $\phi(g)h$ is independent of the choice of $\sigma_g$.

Since $\rho$ is a representation and the identifications in Proposition \ref{prop: iso:VB} are by group isomorphisms, $\phi$ is a groupoid homomorphism. In more detail, let $(g_1, g_2)\in G^2$ and $\sigma_i\in \Bis(\cG)$ such that $\sigma_i(\beta(g_i))=g_i$ for $i\in \{1,2\}$. Then it is easy to check that $(\sigma_1\star\sigma_2)(\beta(g_1g_2))=g_1g_2$ and since $\beta(g_1g_2)=\beta(g_1)$ and $\alpha(g_1g_2)=\alpha(g_2)$, we see for every $h\in E_{\alpha(g_2)}$ and $\xi\in\M$ with $\xi(\alpha(g_2))=h$ that
\begin{align*}
\phi(g_1g_2)h &=\big(\rho(\sigma_1\star\sigma_2)\xi\big)(\beta(g_1)) =\big(\rho(\sigma_1)\rho(\sigma_2)\xi\big)(\beta(g_1))\\
&=\big(\phi(g_1)\big(\phi(g_2)h\big).
\end{align*}
The last equality follows from the fact that $\big(\rho(\sigma_2)\xi\big)(\alpha(g_1))=\big(\rho(\sigma_2)\xi\big)(\beta(g_2))=\phi(g_2)h$ and $\sigma_1(\beta(g_1))=g_1$. Similar computations yield the compatibility with units and inversion, whence $\phi \colon \cG \rightarrow \Phi (E)$ is a morphism of groupoids. Thus we obtain a representation of $\cG$ if $\phi$ is smooth. We postpone the proof of smoothness to Lemma \ref{lem: sm:back} below.

Finally  we  check that $\rho=\rhoG $.  To do this, let $\sigma\in \Bis(\cG)$ and $\xi\in \M$ be arbitrary. By Proposition
\ref{prop: construct1} we have $\rho_\phi (\sigma)(\xi)(m) = \phi(\sigma (m))(\xi(\alpha \circ \sigma (m))$. Now the definition of $\phi$ (with $g=\sigma (m)$ and $h=\xi(\alpha\circ \sigma(m))$, i.e.\ we can choose $\sigma_g \coloneq g$ and $\xi_h\coloneq \xi$) yields the formula
$$
(\rhoG (\sigma)\xi)(m)=\phi\big(\sigma(m)\big)\Big(\xi(\alpha\circ\sigma(m))\Big)=\rho(\sigma).\xi ((\rho_\phi)_M.(\alpha\circ \sigma (m)) = \rho(\sigma).\xi (m),$$
where the last equality follows from the above by $(\rho_\phi)_M (\beta (\sigma (m)) = \alpha (\sigma (m))$ (and we note that $(\rho_\phi)_M.x= (\rho_\phi)_M^{-1}(x)$ for all $x\in M$. We conclude that $\rho=\rhoG$.
 \end{proof}

\begin{lem}\label{lem: sm:back}
Let $\cG =(G\toto M)$ be a Lie groupoid which has enough bisection, $(E, \pi, M)$ a smooth vector bundle and $\rho \colon \Bis(\cG) \rightarrow \SL\big(\Gamma_c(E)\big)$ a joint smooth representation which is local as in Remark \ref{locality}, then we obtain a smooth map
$$\phi\colon \cG\rightarrow \Phi(E), \quad \phi(g)h\coloneq (\rho(\sigma_g)\xi_h)(\beta(g)).$$
\end{lem}

\begin{proof}
Clearly it suffices to test smoothness (and continuity) of $\phi$ locally around a given $g \in G$.
Fixing $g$, we use the local triviality of the bundle $E$ to obtain a pair of bundle trivialisations $\kappa_i \colon \pi^{-1}(U_i) \rightarrow U_i \times F$ $i \in \{1,2\}$ such that $\alpha(\phi (g)) \in U_1$ and $\beta(\phi(g))\in U_2$. Since $\phi$ is a morphism of groupoids, the definition of the frame groupoid shows that the open $g$-neighbourhood $V \coloneq \alpha^{-1} (U_1) \cap \beta^{-1} U_2) \subseteq G$ satisfies $\phi(V) \subseteq \text{im}\, \alpha_{\Phi(E)}^{-1} (U_1) \cap \beta_{\Phi(E)}^{-1} (U_2)$.
Thus it suffices to establish smoothness of $$\tilde{\phi} \colon V \rightarrow U_2 \times \GL (F) \times U_1, \quad g \mapsto (\beta(g), \kappa_2(\beta(g), \cdot) \circ \phi(g) \circ \kappa_1(\alpha(g),\cdot), \alpha(g)).$$
Recall that $\GL(F)$ is an open subset of the finite-dimensional vector space $L(F,F)$. Since on $L(F,F)$ the operator norm topology coincides with the compact open topology (induced by the canonical inclusion into $C(F,F)$), we deduce from the exponential law \cite[Theorem A]{alas2012} that a mapping $f \colon G \rightarrow \GL(F)$ is smooth if and only if the associated map $f^\wedge \colon G \times F \rightarrow F, f^\wedge (g,h) \coloneq f(g)(h)$ is smooth.
Since the bundle trivialisations are smooth, the map $\tilde{\phi}$ (and thus $\phi$) will be smooth if $V \times E|_{U_1} \rightarrow E|_{U_2}, (g,v) \mapsto \phi (g).v$ is smooth.
Recall from \eqref{indrep gpd:Bis} the formula
\begin{align*}
 \phi(g).v = \rho(\sigma_g).\xi_h (\rho(\sigma_g)_M.\alpha (g)) = \rho(\sigma_g).\xi_h (\beta (g)) = \ev_\Gamma (\rho(\sigma_g).\xi_h,\beta(g)),
\end{align*}
where $\sigma_g \in \Bis (\cG)$ with $\sigma_g (\beta(g))=g$, $\xi_h (\alpha (g)) = h$ and $\ev_\Gamma \colon \Gamma_c (E) \times M \rightarrow E$, denotes the evaluation.
Now $\rho \colon \Bis (\cG) \times \Gamma_c (E) \rightarrow \Gamma_c (E)$ and $\ev_\Gamma$ are smooth.
Thus $\phi(g).v$ will be smooth in $(g,v)$ if we can show that $\sigma_g$ and $\xi_h$ depend smoothly on $(g,h)$.
Combining Corollary \ref{cor: sub:fin} with the fact that $\cG$ has enough bisections, $\ev \colon \Bis (\cG) \times M \rightarrow G,\ (\sigma ,m)\mapsto \sigma (m)$ is a smooth surjective submersion.
The definition of a bisection entails that by shrinking $V$, we may assume (cf.\ \cite[Lemma 1.7]{Glofun}) that there is a smooth section of $\ev$ of the form $(S_1,\beta|_V) \colon V \rightarrow \Bis (\cG) \times M$, i.e.\ $\ev (S_1(x),\beta(x)) =x$.
Similarly Corollary \ref{cor: sub:fin} shows that $\ev_\Gamma$ is a surjective submersion, whence there is an open $h$-neighbourhood $W_h \subseteq M$ together with a smooth section $(S_2,\pi|_{W_h}) \colon W_h \rightarrow \Gamma_c (E) \times M$ of $\ev_\Gamma$.
Since the definition of $\phi$ does neither depend on the choice of $\sigma_g$ nor on the choice of $\xi_h$, we can plug in the sections to see that
\begin{align*}
\phi(g).h = \ev_\Gamma (\rho(S_1(g)).S_2 (h),\beta(g)), \quad (g,h) \in V\times W_h
\end{align*}
 is smooth in $(g,h)$. Covering $V\times E|_{U_1}$ with open sets of the form $V\times W_h$, we deduce that $(g,h) \mapsto \phi(g).h$ is smooth.
\end{proof}

\begin{rem}
Comparing our results in this section to \cite{MR2844451} who works with continuous representations of topological groupoids, one sees that in the differentiable category, there is no need for the additional assumption of the groupoid to be "bisectional"\footnote{\cite[p.25]{MR2844451} In the context of the cited paper this puts an additional condition on the topology of the space of continuous mappings $C (M,G)$ which seems to be hard to check in practice.} made in loc.cit.
We remark that the additional ingredient allowing us to avoid this additional assumption was the existence of smooth sections for the evaluation maps.
As long as one considers semi-linear representations of bisections on spaces of smooth bundle sections, this argument carries over without any change to continuous representations.
Thus our results can be seen as an answer to the question in \cite{MR2844451} whether  the additional assumption can be avoided for Lie groupoids.

Further, we remark that our construction explains how to characterise the smoothness of Lie groupoid representations in terms of the smoothness of representations of its bisection group (which was an open question in \cite{MR1958838}).
\end{rem}

\begin{ex}
Returning to Example \ref{ex: gpd:rep}, let us interpret the representations constructed in this section for the examples:
\begin{enumerate}
\item For the Lie groupoid $G \toto \{*\}$, we have $\Bis (\cG) = G$, whence the construction is just the identity (mapping representations of Lie groups to themselves)
\item We have seen in Example \ref{ex: gpd:rep} (b) that the pair groupoid admits only representation on trivial bundles. It is well known (e.g.\ \cite{MR3351079}) that the bisections of the pair groupoid can be identified as $\Bis (\mathcal{P} (M)) \cong \Diff (M)$. Thus if $M$ admits enough bisections (e.g.\ if $M$ is connected, cf.\ \cite{ZHUO,MR3573833}), all joint smooth semi-linear local representations of $\Diff (M)$ on spaces of compactly supported sections correspond to representations of the pair groupoid, i.e.\ such a representation can only exist on the sections of a trivial bundle.
 \item In Example \ref{ex: gpd:rep} (c), if $M$ is the circle $\mathbb{S}^1$ and $G$ is a Lie group, then $\Bis(\cG)$, the group of bisections of the Lie groupoid $\cG=(G\times M\rightrightarrows M)$ is the loop group $C^{\infty}(\mathbb{S}^1, G)$ with its usual topology cf.\ \cite{MR2844451}. Recall that for $G$ a compact semisimple Lie group and $X$ a Riemannian manifold, the representation theory of the group $C^{\infty}(X,G)$ is studied in \cite{GGV77}.
\end{enumerate}
\end{ex}

As a final example, we discuss adjoint representations of a class of Lie groupoids which arises naturally in the context to partial differential equation, Lie pseudogroups and differential Galois theory (cf.\ e.g.\ \cite{Lorenz09,Yudi16,MR0211421}).

\begin{setup}[The jet groupoids and its bisections]\label{setup: jetgpd}
 Let $\cG = (G\toto M)$ be a Lie groupoid and denote by $\Bis_{\text{loc}}(\cG)$ the pseudogroup of local bisections (i.e.\ local sections $\sigma$ of $\beta$ such that $\alpha\circ \sigma$ is a diffeomorphism onto its image). We define the \emph{$k$th jet groupoid} of $\cG$ as the Lie groupoid $J^k(\cG) \coloneq (J^k (\cG) \toto M)$ given by
 $$J^k (\cG) := \{j_x^k \sigma \quad \text{ is the }k\text{th jet expansion\footnotemark  of } \sigma \in \Bis_{\text{loc}}(\cG) \text{ at } x \in \text{dom } \sigma\},$$\footnotetext{i.e.\ the map $m\mapsto j^k_m\sigma$ assigning to each $m\in U$ the equivalence class of the $k$-Taylor expansion of $s$, see \cite[Section 1]{MR702720} for details.}
 with source and target maps $\alpha_{j^k}(j^k_x \sigma) = \alpha \circ \sigma (x)$ and $\beta_{j^k}(j^k_x\sigma) = \beta \circ \sigma (x) = x$. The pair $\big(j^k_x \sigma, j^k_y \tau\big)$ is composable if $\alpha \circ \sigma (x)= y$ and the product is $j^k_x \sigma .j^k_y \tau\coloneq j^k_x (\sigma \star \tau)$ (cf.\ Definition \ref{defn: bis}). This makes sense, since one computes
\begin{align*}
 \alpha_{j^k}\left(j^k_x(\sigma\star\tau)\right)&=\alpha \circ (\sigma\star\tau) (x)
 =\alpha \circ (\sigma \cdot \tau \circ \alpha\circ\sigma)(x)
 \\
 &=(\alpha\circ\tau)\circ (\alpha\circ\sigma)(x)
 =(\alpha\circ\tau)(y) =\alpha_{j^k}(j^k_y\tau),
\end{align*}
   and $\beta_{j^k}\big(j^k_y (\sigma\star\tau)\big)=y=\beta_{j^k}(j^k_y \sigma)$.
 The smooth structure which turns $J^k(\cG)$ into a manifold (and the $k$th jet groupoid into a Lie groupoid) is inherited from the jet bundle \cite[Theorem 1.10]{MR702720}.
The projections $j_m^k \sigma \mapsto j_m^{k-1} \sigma$ induce a sequence of Lie groupoids (cf.\ \cite{Yudi16})
$$\cdots \rightarrow J^k(\cG) \rightarrow J^{k-1}(\cG) \rightarrow \cdots \rightarrow J^0(\cG) \cong \cG.$$
The bisection group $\Bis (J^k (\cG))$ contains the group $\Bis (\cG)$ as a subgroup. Namely, a direct computation shows that
$$h^k \colon \Bis (\cG) \rightarrow \Bis (J^k(\cG)), \quad \sigma \mapsto j^k \sigma, \quad \text{where } j^k\sigma (x) \coloneq j^k_x\sigma,$$
is a group morphism. Since $J^k (\cG)$ is a submanifold of the $k$th jet bundle and $\Bis (\cG)$ is a submanifold of $C^\infty (M,G)$, \cite[Proposition 11.1]{michor1980} implies that $h_k$ is smooth, whence a Lie group morphism.
A bisection in the image $H^k \coloneq h^k (\Bis (\cG)) \subseteq \Bis (J^k (\cG))$ is called \emph{holonomic bisection}, \cite[Section 3.1]{MR3318255} and \cite[Section 1.3.2 and Example 3.1.12]{Salazar13}. Holonomic sections are studied in connection with so called Pfaffian groupoids arising in the study of partial differential equations. By \cite[Lemma 1.3.4]{Salazar13} the Cartan form $\kappa$ on $J^k(\cG)$ detects holonomic bisections in the sense that a bisection $\Sigma$ is holonomic if and only if the pullback of $\kappa$ vanishes $\Sigma^*\kappa =0$.
\end{setup}

In the next example we will discuss the adjoint representation of the jet groupoid and the induced representation on the bisection group of $\cG$.

\begin{ex}[Adjoint representation of the jet groupoid]\label{ex: repjt}
We denote by $\mathcal{A}^k \coloneq \mathcal{A} (J^k(\cG))$ the Lie algebroid associated to the kth jet groupoid (cf.\ \cite{Mackenzie05}). Then $\mathcal{A}^k$ is called the \emph{$k$th jet algebroid}. We note that $\mathcal{A}^0 = \mathcal{A} (J^0(\cG)) = \mathcal{A} (\cG)$ is the Lie algebroid associated to $\cG$. Further, \cite[Section 1.8]{Yudi16} establishes an isomorphism of Lie algebroids $\mathcal{A}^k \cong (J^k \mathcal{A}(\cG) \rightarrow M))$, where $(J^k \mathcal{A}(\cG) \rightarrow M)$ is the $k$th jet bundle of the vector bundle $\mathcal{A}(\cG) \rightarrow M$ (cf.\ \cite[Section 1]{michor1980}).

There is a natural linear action of the jet groupoid $J^k(\cG)$ on $\mathcal{A}^{k-1}$. Let $x\in M, a_x\in (\mathcal{A}^{k-1})_x$ and choose a path $t \mapsto j^{k-1}_x\tau_t \in \alpha_{j^{k-1}}^{-1}(x)$ with $\left.\frac{d}{dt}\right|_{t=0}j^{k-1}_x\tau_t=a_x$. We set $\gamma_t \coloneq \beta_{j^k}(j^{k-1}_x\tau_t)$ to obtain a fibre-wise linear map $\mathcal{A}^{k-1}_{\alpha(\sigma(x))} \rightarrow \mathcal{A}^{k-1}_x$ defined via
\begin{equation}\label{eq: linact}
j^k_x\sigma.a_x \coloneq \left.\frac{d}{dt}\right|_{t=0}\big(j^{k-1}_{\gamma_t}\sigma\big)^{-1} \cdot \big(j^{k-1}_x\tau_t\big) \cdot \big( j^{k-1}_x\sigma\big).
\end{equation}
One can prove, \cite[Chapter 2.3]{Yudi16} that \eqref{eq: linact} yields a groupoid morphism
$$\mathrm{Ad} \colon J^k(\cG)\rightarrow \Phi(A^{k-1}), \quad \mathrm{Ad}(j^k_x\sigma)(a_x)=j^k_x\sigma.a_x,$$
called (for obvious reasons) the \emph{adjoint representation} of the $k$th jet groupoid on the $(k-1)$st jet bundle. Following Proposition \ref{prop: construct1} we obtain an induced semi-linear representation of $\Bis (J^k(\cG))$ on the sections of $\mathcal{A}^{k-1}$ and using the Lie group morphism $h^k \colon \Bis (\cG) \rightarrow \Bis (J^k (\cG), \sigma \mapsto j^k\sigma$ we obtain an induced representation of the holonomic sections
     $\rho_{\mathrm{Ad}, \text{hol}} \colon \Bis(\cG) \rightarrow \Gamma_c(\mathcal{A}^{k-1})$ defined via
    \begin{align*}
   \rho_{\mathrm{Ad},\text{hol}}(\sigma)(\xi)(m)& \coloneq  \rho_{\mathrm{Ad}}(h^k(\sigma))(\xi)(m) = \mathrm{Ad} (j^k \sigma(m))\Big(\xi (\alpha_{j^k}\circ j^k\sigma (m))\Big)\\
     &=\mathrm{Ad}(j^k_m\sigma)\Big(\xi(\alpha\circ \sigma(m)) )\Big), \quad \xi \in \Gamma_c(\mathcal{A}^{k-1}), m\in M
     \end{align*}
 We specialise now to $k=1$ and recall that $\mathcal{A}^0 = \mathcal{A} (\cG)$. Now $\Bis (\cG)$ acts on $G$ by conjugation via $C_\sigma (g) := \sigma(\beta(g))^{-1}\cdot g \cdot \sigma (\alpha (g))$.
 This allows one to identify \eqref{eq: linact} for $k=1$ with the derivative of the conjugation action, i.e.\ we obtain for $k=1$ the formula
 \begin{equation}\label{eq: adjoint}
 \rho_{\mathrm{Ad}, \text{hol}} (\sigma)(\xi)(m) = T_{\alpha \circ \sigma (m)} C_{\sigma} (1_{\alpha \circ \sigma (m)}) (\xi (\alpha \circ \sigma (m)), \ \text{for } m\in M, \xi \in \Gamma_c(\mathcal{A}(\cG)),
 \end{equation}
 (cf.\ \cite[Section 1.3.2]{Salazar13}). One can identify $\Gamma_c (\mathcal{A}(\cG))$ as the Lie algebra of $\Bis (\cG)$ (see e.g.\ \cite{MR3351079}, the proof for non-compact $M$ is similar). Using this identification, one computes that the adjoint action of $\Bis (\cG)$ on its Lie algebra coincides with \eqref{eq: adjoint}. In conclusion, the adjoint action of $J^1(\cG)$ induces on the holonomic bisections, aka.\ $\Bis (\cG)$, the adjoint action of the Lie group of bisections.
\end{ex}

\subsection*{\texorpdfstring{Representations of $\SG (\alpha)$}{Representations of the group of self-mappings}}
\addcontentsline{toc}{subsection}{Representations of the group of self-mappings}
In this section we shift our focus from the group of bisections to the group $\SG(\alpha)$.
Similar to the last section we will establish a correspondence between smooth representations of the groupoid $\cG$ and the joint smooth (semi-linear and local) representations of the group $\SG(\alpha)$.
Again every smooth representation $\RepG \colon \cG \rightarrow \Phi(E)$ induces a semi-linear and local joint smooth representation $$\rho_{\phi,S}\colon \SG(\alpha)\rightarrow \SL(\Gamma_c(\alpha^*E)).$$
We will further investigate the restriction of $\rho_{\phi,S}$ to the subgroup $\alpha_*(\Bis(G))=\{\sigma\circ\alpha\colon \sigma\in \Bis(G)\ \}$. If $\cG$ is $\alpha$-proper, the restriction induces a semi-linear, local and joint smooth representation $\Bis(\cG)\rightarrow \SL(\Gamma_c(E))$ which we denote by $\rho_{\phi,B}$.
Finally, we discuss how one can associate to a semi-linear, local and joint smooth representation of $\SG(\alpha)$ a corresponding representations of the Lie groupoid $\cG$.

\begin{prop}\label{prop: construct2}
Let $\RepG \colon G \rightarrow \Phi(E)$ be a representation of the Lie groupoid $\cG = (G\toto M)$ on the vector bundle $(E,\pi, M)$.
Then $\RepG$ induces a joint smooth semi-linear and local representation
$\rho_{\phi,S} \colon \SG(\alpha) \rightarrow \SL (\Gamma_c(\alpha^*E))$ via
\begin{equation}\label{eq: const:SGREP}
\rho_{\phi,S} (f) (\xi)(x) \coloneq \RepG\big(f(x)\big)\xi\big(R_f(x)\big),\ \ f\in \SG(\alpha),\  \xi\in \Gamma_c(\alpha^*E), \text{ and } x\in G,\end{equation} where $R_f(x)=xf(x)$.
\end{prop}

\begin{proof}
Let us first prove that $\rho_{\phi,S}$ yields linear mappings $\Gamma_c (\alpha^* E)) \rightarrow \Gamma_c (\alpha^* E))$.
Note that $\xi(R_f(x))=\xi(xf(x))\in E_{\alpha(xf(x))}=E_{\alpha(f(x))}$ and $\RepG\big(f(x)\big): E_{\alpha(f(x))}\rightarrow E_{\beta(f(x))}=E_{\alpha(x)}$.
Therefore $\big(\rho_{\phi,S} (f)\xi\big)(x)=\RepG\big(f(x)\big)\xi\big(xf(x)\big)\in E_{\alpha(x)}$, this means that $\rho_{\phi,S} (f)\xi$ is a section on the fibre bundle $(\alpha^*E, \alpha^*\pi, G)$.
Now if $\xi\in\Gamma_c (\alpha^*E)$ and $K\coloneq\text{supp}(\xi)$ then
$$\{x\in G \colon \rho_{\phi,S} (f) (\xi)(x)\notin  0_G\}\subseteq R_f^{-1}(K).$$
Since $R_f\in \Diff(G)$, $R_f^{-1}(K)$ is compact and thus $\rho_{\phi,S} (f) (\xi)\in \Gamma_c (\alpha^*E)$.
To establish smoothness of $\big(\rho_{\phi,S} (f)\xi\big)$ in $x$, observe that \eqref{eq: const:SGREP} arises as a composition of the following smooth mappings:
\begin{itemize}
\item $G\rightarrow G\times_{\alpha}E_{\pi}, \quad x\mapsto\big(f(x),\xi\circ R_f(x)\big)$ and
\item the linear action $\Phi \colon G\times_{\alpha}E_{\pi}\rightarrow E_{\pi}$ with $\Phi(g,x)=\RepG(g)x$.
\end{itemize}
 Thus $\big(\rho_{\phi,S} (f)\xi\big)\in\Gamma_c(\alpha^*E)$ for every $\xi\in\Gamma_c(\alpha^*E)$ and clearly $\rho_{\phi,S}(f)$ is linear in $\xi$.

It is easy to check that $\rho_{\phi,S} (\alpha).\xi=\xi$ for every $\xi\in\Gamma_c(\alpha^*E)$, i.e.\ $\rho_{\phi,S} (\alpha)$ is the identity on $\Gamma(\alpha^*E)$ (and we recall that the source map $\alpha$ is the identity of the Lie group $\SG(\alpha)$).
Let us now check that $\rho_{\phi,S} (f\star g) = \rho_{\phi,S} (f) \circ \rho_{\phi,S}(g)$ for arbitrary $f, g\in\SG(\alpha)$.
Consider $\xi\in\Gamma_c(\alpha^*E)$ and $x\in G$ and compute as follows
\begin{align*}
\Big(\rho_{\phi,S} (f\star g)\xi\Big)(x)&=\RepG\big((f\star g)(x)\big)\xi\big(R_{f\star g}(x)\big) =\RepG\Big(f(x)g\big(R_f(x)\big)\xi\Big(R_g\circ R_f (x)\Big)\\
&=\RepG(f(x))\RepG\big(g(R_f(x))\big)\xi\Big(R_g\big(R_f(x)\big)\Big)
=\RepG(f(x))\Big(\rho_{\phi,S} (g)\xi\Big)\big(R_f(x)\big)\\
&=\rho_{\phi,S} (f)\Big(\rho_{\phi,S} (g)\xi\Big)(x) =\Big(\rho_{\phi,S} (f)\circ\rho_{\phi,S} (g)\xi\Big) (x).
\end{align*}
Therefore $\rho_{\phi,S}$ defines a group homomorphism $\SG(\alpha) \rightarrow \GL (\Gamma_c (\alpha^* E))$ and we will check in Remark \ref{Semilinearity1} below that $\rho_{\phi,S}$ takes its image in $\SL(\Gamma_c (\alpha^* E))$. A direct computation involving \eqref{eq: const:SGREP} shows that $\rho_{\phi,S}$ is local in the sense of Definition \ref{locality1}.

Finally, we have to establish joint smoothness of the semi-linear representation $\rho_{\phi,S}$. Again the proof in this case is similar to the case of the induced $\Bis(\cG)$-representation in Proposition \ref{prop:smooth}:
We just note that by \cite[Remark 2.10]{AS17}, $R \colon \SG(\alpha) \rightarrow R_{\SG(\alpha)} \subseteq \Diff (G), f\mapsto R_f$ is a diffeomorphism onto a closed Lie subgroup. Now $R_{\SG(\alpha)}$ is a submanifold of $\Prop(G,G)$, whence we obtain a composition of smooth mappings
$$\rho_{\phi,S}(f)(\xi)=\ActG\circ(f, \xi\circ R_f)=\ActG_*\circ(f, \text{Comp}(\xi,R_f)).$$
Here $\ActG$ is the linear action of $\cG$ on $E$ induced by $\RepG$ (cf.\ Remark \ref{rem: linRep}). In conclusion $\rho_{\phi,S}$ is joint smooth.
\end{proof}

\begin{rem}[Semilinearity of the induced action]\label{Semilinearity1}
In the situation of Proposition \ref{prop: construct2}, semi-linearity $\rho_{\phi,S}(f)$ for $f\in\SG(\alpha)$ can be checked as follows: Consider the map
$$\rho_{\phi,S}^G(f) \colon C^{\infty}(G)\rightarrow C^{\infty}(G), \quad \rho_{\phi,S}^G(f)(h)=h\circ R_f$$
is an algebra isomorphism (obviously induced by $(\rho_{\phi,S})_G \coloneq R_f \in \Diff (G)$).
We thus obtain semi-linearity, by the following, now trivial, observation:
$$\rho_{\phi,S}(f)(h\xi)=\rho_{\phi,S}^G(f)(h)\rho_{\phi,S}(f)\xi, \quad h \in C^\infty (G), f\in \SG(\alpha), \xi \in \Gamma_c (\alpha^*E).$$
\end{rem}

\begin{rem}\label{embedded}
The pushforward $\alpha^*$ induces injective $\R$-linear maps
\begin{align*}
C^{\infty}(M) \rightarrow C^\infty (G) ,\quad f \mapsto f \circ \alpha ; \qquad
\Gamma(E)\rightarrow \Gamma(\alpha^*E),\quad \xi \mapsto \xi\circ\alpha.
\end{align*}

If $\cG$ is $\alpha$-proper, i.e.\ $\alpha$ is a proper map, these maps are smooth with respect to the function space topologies and restrict to a smooth map $\psi \colon \Gamma_c (E) \rightarrow \Gamma_c (\alpha^* E)$.
Recall from \cite[Theorem 4.11]{AS17} that the restriction yields a group monomorphism
  $$\Psi \colon \Bis (\cG) \rightarrow \SG (\alpha), \quad \sigma \mapsto \sigma \circ \alpha$$
 which is even a Lie group morphism for $\alpha$-proper groupoids. Note that by \cite[Example 4.13]{AS17} in general $\Psi(\Bis(\cG))\neq \SG(\alpha)$.
\end{rem}

  In the rest of the paper we denote by $\rho_{\RepG,B}$ the representation on $\Bis(\cG)$ induce by a representation $\phi\colon \cG\rightarrow \Phi(E)$ (see Proposition \ref{prop: construct1}). Similarly, the representation induced on $\SG(\alpha)$, introduced in Proposition \ref{prop: construct2}, is denoted by $\rho_{\RepG,S}$. The following result relates these representations.

\begin{prop}\label{prop:construct3}
Let $\RepG \colon \cG \rightarrow \Phi(E)$ be a smooth representation of an $\alpha$-proper Lie groupoid $\cG$.
Then $\rho_{\phi,S}$ induces a joint smooth, semi-linear and local representation $\rho_{B}\colon \Bis(\cG)\rightarrow \SL\big(\Gamma_c(E)\big)$ via
\begin{equation}\label{eq: res:ind}
\big(\rho_{B}(\sigma)\xi\big)(\alpha(x))=\big(\rho_{\phi,S}(\Psi(\sigma))\psi(\xi)\big)(x) \text{ for } \sigma\in \Bis(\cG), \xi\in \Gamma_c(E),\end{equation}
where $\Psi$ and $\psi$ are the maps from Remark \ref{embedded}. Furthermore, if $\cG$ has enough bisections, then $\rho_{B}=\rho_{\phi,B}$.
\end{prop}
\begin{proof}
Since $\cG$ is $\alpha$-proper, $\psi$ is a smooth mapping which takes compactly supported sections of $E$ to compactly supported sections of $\alpha^*E$. It is now easy to check that $\rho_{\RepG,S}(\Psi(\sigma))\psi(\xi)$ is an element of $\Gamma_c(\alpha^*E)$ which is fixed on the $\alpha$-fibres, so the definition of $\big(\rho_{B}(\sigma)\xi\big)(\alpha(x))$ makes sense.
Since $\Psi \colon \Bis(\cG)\rightarrow \SG(\alpha)$ is a group morphism,
\begin{align*}
\big(\rho_{B}(\sigma_1\star\sigma_2)\xi\big)(\alpha(x))&=\Big(\rho_{\phi,S}\big(\Psi(\sigma_1\star\sigma_2)\big)(\psi(\xi))\Big)(x)\\
&=\Big(\rho_{\phi,S}(\Psi(\sigma_1))\rho_{\phi,S}(\Psi(\sigma_2))(\psi(\xi))\Big)(x)\\
&=\Big(\rho_{\phi,S}(\Psi(\sigma_1))(\rho_{B}(\sigma_2)\xi)\circ\alpha\Big)(x)
=\Big(\rho_{B}(\sigma_1)\rho_{B}(\sigma_2)\xi\big)\Big)(\alpha(x)).
\end{align*}
By linearity of $\rho_{\phi,S}$ and the above computation, $\rho_{B}$ is a representation of $\Bis (\cG)$ on $\Gamma_c(E)$. We will now prove that this representation is semi-linear and local. For locality, let $\sigma(\alpha(x))=\alpha(x)$ then for every $\xi\in \Gamma_c(E)$ we derive locality as follows
\begin{align*}
\big(\rho_{B}(\sigma)\xi\big)(\alpha(x))&=\Big(\rho_{\RepG,S}\big(\Psi(\sigma)\big)\psi(\xi)\Big)(x)
=\RepG\Big(\big(\sigma\circ\alpha\big)(x)\Big)(\xi\circ\alpha)\big(x(\sigma\circ\alpha)(x)\big)\\
&=\RepG\Big(\sigma\big(\alpha(x)\big)\Big)\xi\big(\alpha\big(x\sigma(\alpha(x))\big)\big)
=\RepG\Big(\alpha(x)\Big)\xi\big(\alpha(x)\big)
=\xi\big(\alpha(x)\big).
\end{align*}
To establish semi-linearity, we recall from Remark \ref{Semilinearity1} that   $\rho_{\RepG,S}$ is semi-linear with respect to the associated diffeomorphism
$$\rho_{\RepG,S}(f)(w\xi)=R_f^* (w)\rho_{\RepG,S}(f)\xi\ \ \mbox{for every}\ \ w\in C^{\infty}(G), f\in \SG(\alpha)\ \mbox{and}\ \xi\in \Gamma_c(\alpha^*E).$$
Now  let $w\in C^{\infty}(M)$, $\sigma\in \Bis(\cG)$ and $\xi\in\Gamma_c(E)$.
\begin{align*}
\big(\rho_{B}(\sigma).w\xi\big)(\alpha(x))&=\Big(\rho_{\RepG,S}(\Psi(\sigma))\psi(w\xi)\Big)(x)
=\Big(\rho_{\RepG,S}(\sigma\circ\alpha)\big((w\circ\alpha)(\xi\circ\alpha)\big)\Big)(x)\\
&=R_{\sigma\circ\alpha}^*(\psi(w))\Big(\rho_{\RepG,S}(\sigma\circ\alpha)\xi\circ\alpha\Big)(x)\\
&= (w \circ \alpha \circ \sigma)\cdot\Big(\rho_{B}(\sigma)\xi\big)(\alpha(x)).
\end{align*}
For the last equality we have used $R_{\sigma\circ\alpha}^* (\psi (w)) (x)=w \circ \alpha\Big(x\big(\sigma\circ\alpha)(x)\Big)=(w\circ\alpha\circ\sigma)(\alpha(x))$. Since $\alpha \circ \sigma \in \Diff (M)$ we see that $\rho_B$ is semilinear with associated algebra isomorphism $\rho_{B}(\sigma)^M\colon C^{\infty}(M)\rightarrow C^{\infty}(M), w \mapsto (\alpha\circ \sigma)^*w$.
Since $\rho_{\phi,S}$ is joint smooth and $\Psi,\psi$ are smooth (since $\cG$ is $\alpha$-proper), \eqref{eq: res:ind} shows that $\rho_B$ is joint smooth.

Now Proposition \ref{prop:construction3} entails that $\rho_{B}$ induces a representation $\tilde{\phi}\colon G\rightarrow\Phi(E)$ via
$$\tilde{\phi} (g).h=\big(\rho_{B}(\sigma)\xi\big)(\beta(g)), \ \ \mbox{where}\ \ \sigma(\beta(g))=g, \mbox{and}\ \ \xi(\alpha(g))=h.$$
However, inversion in the groupoid intertwines source and target mappings, whence the definition of $\rho_{B}$ leads to
$$\begin{array}{ll}
\big(\rho_{B}(\sigma)\xi\big)(\beta(g))&=\Big(\rho_{\RepG,S}(\sigma\circ\alpha)(\xi\circ\alpha)\Big)(g^{-1})\\
&=\RepG\big((\sigma\circ\alpha)(g^{-1})\big)(\xi\circ\alpha)\big(g^{-1}(\sigma\circ\alpha)(g^{-1})\big)\\
&=\RepG(g)\xi(\alpha(g))=\RepG(g).h.
\end{array}
$$
Therefore, $\tilde{\phi}=\RepG$ and consequently $\rho_{B}=\rho_{\phi, B}$.
\end{proof}

The authors were not able to prove that every semi-linear representation of $\SG(\alpha)$ is induced by a representation of the underlying Lie groupoid (cf.\ Proposition \ref{prop: construct2}). Instead, we have the following (whose proof is similar to the one of Proposition \ref{prop:construction3}):

\begin{cor}\label{locality3}
Let $\cG$ be an $\alpha$-proper Lie groupoid and $\rho_{S} \colon \SG(\alpha) \rightarrow \SL (\Gamma_c(\alpha^*E))$ be a joint smooth, semi-linear representation, then
$\rho_{B}\colon \Bis(\cG)\rightarrow \SL\big(\Gamma_c(E)\big)$ with
$$\big(\rho_{B}(\sigma)\xi\big)(\alpha(x)):=\Big(\rho_{S}(\Psi(\sigma))\psi(\xi)\Big)(x)$$
is a joint smooth semi-linear  representation of $\Bis(\cG)$.
If $\rho_{S}$ is in addition a local representation of $\SG(\alpha)$, then $\rho_{B}$ is a local representation and therefore Proposition \ref{prop: construct1} implies $\rho_{B}=\rho_{\phi,B}$ for some representation $\phi$ of $\cG$.
\end{cor}

The representation $\phi \colon \cG \rightarrow \Phi(E)$ associated to $\rho_S$ in Corollary \ref{locality3} induces a semi-linear representation $\rho_{\phi, S}$ of $\SG(\alpha)$. In the following, we show that the restriction of two sections $\rho_{S}(f)\xi$ and $\rho_{\phi, S}(f)\xi$ on $M$ are equal, where $f\in \SG(\alpha), \xi\in \Gamma_c(\alpha^*E)$, note that the base space of the pullback bundle $\alpha^*E$ is $G$ .

\begin{prop}\label{Construct 2}
  Let $\cG$ be an $\alpha$-proper Lie groupoid which has enough bisection and $\rho_{S} \colon \SG(\alpha)\rightarrow \SL\big(\Gamma_c(\alpha^*E)\big)$ be a joint smooth local semi-linear representation, then there is a smooth representation $\phi\colon \cG\rightarrow \Phi(E)$ such that the sections $\rho_{S}(f)\xi$ and $\rho_{\phi, S}(f)\xi$ are equal on $M$, where $f\in \SG(\alpha), \xi\in \Gamma_c(\alpha^*E)$ and $\big(\rho_{\phi, S}(f)\xi)(x)=\phi(f(x))\xi(xf(x))$.
\end{prop}
\begin{proof}
Since $\cG$  has enough bisection, we choose and fix for every $g\in G$ a bisection $\sigma_g\in \Bis(\cG)$ with $\sigma_g(\beta(g))=g$. Define $f_g=\Psi(\sigma_g)$ to obtain an element $f_g\in S_{\cG}(\alpha)$ with $f_g(\beta(g))=g$.
Now \cite[Lemma 11.8]{MR1930277} allows us to choose for $h\in E_{\alpha(g)}$  a section $\xi_h\in \Gamma_c(E)$ with $\xi_h(\alpha(g))=h$.
Further, $\psi(\xi_h)\in \Gamma_c(\alpha^*E)$ and $\psi(\xi_h)(g)=h$.

We define a representation $\cG$ via
$$\phi\colon G\rightarrow \Phi(E), \quad \phi(g)h\coloneq \big(\rho_{S}(f_g)\xi_h\big)(\beta(g)), \text{ where } g\in G, h\in E_{\beta(g)}.
$$
To see that this makes sense, let $\rho_B\colon \Bis(\cG)\rightarrow \SL\big(\Gamma_c(E)\big)$ be the representation from Proposition \ref{prop:construct3} which by Corollary \ref{locality3} is joint smooth, semi-linear and local, whence by Proposition \ref{prop: construct2} induces a  representation $\eta \colon G\rightarrow \Phi(E)$. Now let $g\in G$ and $h\in E_{\alpha(g)}$.
By definition we have $(\sigma_g\circ\alpha)(\beta(g))=g$ (using the identififaction $M \subseteq G$) and $\xi_h (\alpha (g)) = h$ and thus
\begin{align*}
\phi(g)h&=\big(\rho_{S}(\sigma_g\circ\alpha)(\xi_h \circ\alpha)\big)(\beta(g))=\rho_B(\sigma_g)\xi_h(\beta(g))\\
&=\rho_{\eta}(\sigma_g)\xi_h(\beta(g))=\eta(\sigma_g(\beta(g)))\xi_h(\alpha\circ\sigma(\beta(g)))\\
&=\eta(g)\xi_h(\alpha(g))=\eta(g)h.
\end{align*}
Therefore $\phi=\eta$ and $\phi$ is a representation of the Lie groupoid $\cG$
Now let $\rho_{\phi,S}$ be as in Proposition \ref{prop: construct2}, then for $f\in S_{\cG}(\alpha), \xi\in \Gamma_c(\alpha^*E)$ and $m\in M$,
$$\big(\rho_{\phi,S}(f)\xi\big)(m)=\phi(f(m))\xi(mf(m))=\phi(f(m))\xi(f(m))=\phi(g')h',$$
where $g'=f(m)$ and $h'=\xi(f(m))$. Since $f(\beta(g'))=f\big(\beta(f(m))\big)=f(\alpha(m))=f(m)=g'$, and $h'=\xi(g')$,  the definition of $\phi$ implies that $$\phi(g')h'=\big(\rho_{S}(f)\xi\big)(\beta(g'))=\big(\rho_{S}(f)\xi\big)(\beta(f(m)))=\big(\rho_{S}(f)\xi\big)(m).$$
We conclude that $\rho_{\phi,S} (f)\xi (m) = \phi(g').h' = \rho_S (f)\xi(m)$ and thus the sections coincide on $M$.
\end{proof}

\subsection*{Functorial aspects of the construction}
\addcontentsline{toc}{subsection}{Functorial aspects of the construction}

In this section we discuss the functorial aspects of the constructions in the last two sections.
Recall that for a Lie groupoid $\cG = (G\toto M)$ we denote by $\Rep (\cG)$ the category of Lie groupoid representations. In Proposition \ref{prop: construct1} and Proposition \ref{prop:construct3} we have seen that each representation induces a smooth (semi-linear and local) representation of the groups $\Bis (\cG)$ and $\SG(\alpha)$. We will now compute that the construction is even functorial:

\begin{prop}\label{prop: fun}
Let $\delta \colon \phi_1 \rightarrow \phi_2$ be a morphism between representations $\phi_1 \colon \cG \rightarrow \Phi(E),\phi_2 \colon \cG \rightarrow \Phi(F) \in \Rep(\cG)$ and $\rho_{\phi_i}$, $\rho_{\phi_i,S}$ for $i\in \{1,2\}$ the induced representations of $\Bis (\cG)$ and $\SG(\alpha)$.
Then $\delta$ induces morphisms of representations
\begin{align*}
\rho (\delta) \colon \Gamma_c (E) \rightarrow \Gamma_c(F), \quad X \mapsto \delta \circ X\\
\rho_S (\delta) \colon \Gamma_c (\alpha^*E) \rightarrow \Gamma_c(\alpha^*F), Y \mapsto (\alpha^*\delta) \circ Y,
\end{align*}
where $\alpha^*\delta \colon \alpha_*E\rightarrow \alpha^*F$ is the bundle morphism induced by pulling back $\delta$ along $\alpha$.
\end{prop}

\begin{proof}
Note first that $\delta_* \colon \Gamma_c (E) \rightarrow \Gamma_c (F),\ \delta_* (X) = \delta \circ X$ and $(\alpha^*\delta)_* \colon \Gamma_c (\alpha^* E) \rightarrow \Gamma_c (\alpha^*F), \alpha^*\delta)^*(X)=\alpha_*\delta \circ X$ makes sense, as $\delta$ and $\alpha^*\delta$ are bundle morphism over the identity. Clearly $\delta_*$ and $(\alpha_*\delta)^*$ are linear and smooth by \cite[Corollary 10.14]{michor1980}.
Thus we only have to check that the mappings are equivariant with respect to the group actions. However, using equivariance of $\delta$ with respect to the linear action maps $\Phi_1,\Phi_2$ associated to $\phi_1,\phi_2$ one obtains
\begin{align*}
\rho_{\phi_2} (\sigma). (\delta_*(\xi))(m) =& \phi_2 (\sigma (m)).\delta \circ X \circ (\alpha \circ \sigma(m)) = \Phi_2 (\sigma (m) , \delta \circ X \circ (\alpha \circ \sigma(m)) \\ =& \delta ((\Phi_1 (\sigma (m)).(X\circ \alpha \circ \sigma (m)))  \\ =& \delta_* \rho_{\phi_1} (\sigma) (X)(m) \quad \forall \sigma \in \Bis (\cG) , \xi  \in \Gamma_c (M) \text{ and } m\in M.
\end{align*}
Thus $\delta_*$ is a morphism $\rho_{\phi_1} \rightarrow \rho_{\phi_2}$ in $\Rep (\Bis (\cG))$.
For the $\SG(\alpha)$ statement, we observe first that for $\xi \in \Gamma_c(\alpha_*E)$ we have $\xi (R_f (x)) \in E_{\alpha \circ R_f (x)} = E_{\alpha \circ f (x)}$ for all $x\in G$ and $f\in \SG(\alpha)$.
We compute then (with a harmless identification):
\begin{align*}
\rho_{\phi_2,S} (f)((\alpha_*\delta)^*(\xi))(x) =& \phi_2 (f(x)).(\alpha^*\delta)_*\xi(R_f(x))  =  \Phi_2 (f(x),\delta (\xi(R_f(x))))\\
						=& \delta (\Phi_1 (f(x), \xi(R_f(x))) = (\alpha^*\delta)_* \rho_{\phi_1, S} (f)(\xi)(x) \qedhere
\end{align*}
\end{proof}

As every element in a vector bundle is contained in the image of a compactly supported section, for bundle morphisms $\delta \neq \delta'$  we have $\rho(\delta) \neq \rho (\delta')$ (similarly for $\rho_S$).

\begin{cor}
Combining Propositions \ref{prop: construct1}, \ref{prop:construct3} and  \ref{prop: fun} we obtain faithful functors
\begin{align*}
 \rho \colon \Rep(\cG) \rightarrow \Rep (\Bis (\cG)), (\phi \colon \cG \rightarrow \Phi (E)) \mapsto \rho(\phi), (\delta \colon \phi \rightarrow \phi') \mapsto \rho (\delta)\\
  \rho_S \colon \Rep(\cG) \rightarrow \Rep (\SG (\alpha)), (\phi \colon \cG \rightarrow \Phi (E)) \mapsto \rho_S(\phi), (\delta \colon \phi \rightarrow \phi') \mapsto \rho_S (\delta)
\end{align*}
\end{cor}

\subsection*{Acknowledgements}

The authors thank K.--H.\ Neeb for helpful conversations on the subject of this work. We also thank the anonymous referee for numerous comments and suggestions which helped improve the article.

\appendix{

\section{Infinite-dimensional manifolds and manifolds of mappings}\label{App:Mfd}

In this appendix we collect the necessary background on the theory of manifolds
that are modelled on locally convex spaces and how spaces of smooth maps can be
equipped with such a structure. Let us first recall some basic facts concerning
differential calculus in locally convex spaces. We follow
\cite{hg2002a,BGN04}.

\begin{defn}\label{defn: deriv} Let $E, F$ be locally convex spaces, $U \subseteq E$ be an open subset,
$f \colon U \rightarrow F$ a map and $r \in \N_{0} \cup \{\infty\}$. If it
exists, we define for $(x,h) \in U \times E$ the directional derivative
$$df(x,h) \coloneq D_h f(x) \coloneq \lim_{t\rightarrow 0} t^{-1} \big(f(x+th) -f(x)\big).$$
We say that $f$ is $C^r$ if the iterated directional derivatives
\begin{displaymath}
d^{(k)}f (x,y_1,\ldots , y_k) \coloneq (D_{y_k} D_{y_{k-1}} \cdots D_{y_1}
f) (x)
\end{displaymath}
exist for all $k \in \N_0$ such that $k \leq r$, $x \in U$ and
$y_1,\ldots , y_k \in E$ and define continuous maps
$d^{(k)} f \colon U \times E^k \rightarrow F$. If $f$ is $C^\infty$ it is also
called smooth. We abbreviate $df \coloneq d^{(1)} f$ and for curves $c \colon I \rightarrow M$ on an interval $I$, we also write $\dot{c} (t) \coloneq \frac{\dd}{\dd t} c (t) \coloneq dc(t,1)$.
\end{defn}

Note that for $C^r$-mappings in this sense the chain rule holds. Hence for $r \in \N_0 \cup \{\infty\}$, manifolds modelled on locally convex spaces can be defined as usual.
The model space of a locally convex manifold and the manifold as a topological space will always be assumed to be Hausdorff spaces.
However, we will not assume that infinite-dimensional manifolds are second countable or paracompact.
We say $M$ is a \emph{Banach} (or \emph{Fr\'echet}) manifold if all its modelling spaces are Banach (or
Fr\'echet) spaces.

Direct products of locally convex manifolds, tangent spaces and tangent bundles as well as $C^r$-maps between manifolds and vector bundles with locally convex fibre may be defined as in the finite-dimensional setting.
Furthermore, we define \emph{(locally convex) Lie groups} as groups with a $C^\infty$-manifold structure turning the group operations into $C^\infty$-maps. See \cite{neeb2006,hg2002a} for more details.

\begin{setup}\label{fS:top}
For smooth manifolds $M,N$ we let $C^\infty (M,N)$ be the set of smooth mappings $f \colon M \rightarrow N$. If $M$ is a finite-dimensional (not necessarily compact) manifold and $N$ is a locally convex manifold. We write $\CfS (M,N)$ for the space of smooth functions with the fine very strong topology we will describe now.
 This is a Whitney type topology controlling functions and their derivatives on locally finite families of compact sets. Before we describe a basis of the fine very strong topology, we have to construct a basis for the strong topology which we will then refine. To this end, we recall the so called basic neighbourhoods (see \cite{HS2016}) for $f\in C^\infty (M,N)$: Let $A$ be compact, $\varepsilon >0$ and fix a pair of charts $(U, \psi)$ and $(V,\varphi)$ such that $A \subseteq V$ and $\psi \circ f \circ \varphi^{-1}$ makes sense.
 Using standard multiindex notation, define an \emph{elementary $f$-neighborhood}
 $$\mathcal{N}^{r} \left( f; A , \varphi,\psi,\epsilon \right) \coloneq \left\{\substack{\displaystyle g \in C^\infty (M,N), \quad \psi \circ g|_A \quad\text{ makes sense,}\\ \displaystyle\sup_{\alpha \in \N_0^d, |\alpha|< r }\sup_{x \in \varphi (A)}\lVert \partial^\alpha \psi \circ f \circ \varphi^{-1}(x) - \partial^\alpha\psi \circ g \circ \varphi^{-1}(x)\rVert < \varepsilon}\right\}.$$ 
 A basic neighborhood of $f$ arises now as the intersection of (possibly countably many) elementary neighborhoods $\mathcal{N}^{r} \left( f; A_i , \varphi_i,\psi_i,\epsilon_i \right)$ where the family $(V_i,\varphi_i)_{i\in I}$ is locally finite. Basic neighborhoods form the basis of the very strong topology (see \cite{HS2016}). To obtain the fine very strong topology, one declares the sets  
 \begin{equation}\label{eq: cpt:nbhd}
 \{g \in C^\infty (M,N) \mid \exists K \subseteq M \text{ compact such that } \forall x \in M\setminus K,\ g(x) =f(x) \} \tag{$\star$}
 \end{equation}
 to be open and constructs a subbase of the fine very strong topology as the collection of sets $\eqref{eq: cpt:nbhd}$ (where $f \in C^\infty (M,N)$) and the basic neighbourhoods.
  
 The fine very strong topology is an extremly fine topology, which has the advantage that $\CfS (M,N)$ becomes an infinite-dimensional manifold (cf.\ \cite{michor1980} and \cite{HS2016}). However, if $M$ is compact, the fine very strong topology coincides with the familiar compact open $C^\infty$-topology (see \cite[I.5]{neeb2006}). If $N = \R^n$ then the pointwise operations turn $C^\infty_{\text{fS}} (M,\R^n)$ into a locally convex vector space.
\end{setup}

\begin{setup}\label{setup: sections}
Let $q \colon Y \rightarrow X$ be a surjective submersion of smooth finite-dimensional manifolds.
Then we denote by $S_q (X,Y) \subseteq \CfS (X,Y)$ the smooth sections of $q$ (i.e.\ smooth mappings $s$ with $q \circ s = \id_X$).
Recall from \cite[Proposition 10.10]{michor1980} that $S_q (X,Y)$ is a splitting submanifold of $\CfS (X,Y)$.

Mostly we will be interested in the following special case: Let $(E,\pi, M)$ be vector bundle, then we let $\Gamma (E) := S_\pi (M,E)$ be the space of smooth sections of the bundle.
Consider for $f \in \Gamma(E)$ its \emph{support}
$$\text{supp } f=\overline{\{x\in M : f(x)\notin \mathrm{im}\, 0_M\}},$$ where $\bar{A}$ denotes the closure of the set $A$ and $0_M$ is the zero section. We denote the set of all \emph{compactly supported smooth section} by $\Gamma_c(E)$ and recall that the subspace topology turns $\Gamma_c (E) \subseteq S_\pi (M,E) \subseteq \CfS (M,E)$ into an open submanifold of $\Gamma (E)$.
We further recall that with respect to this structure the pointwise operations turn $\Gamma_c (E)$ into a locally convex space. Moreover, the pointwise operations turn $\Gamma (E)$ and $\Gamma_c(E)$ into $C^{\infty}(M)$-modules (though not into topological modules if $M$ is non-compact).
\end{setup}

\begin{rem}
In the above we endowed every function space with the fine very strong topology (respectively the subspace topologies induced by it). However, we wish to consider $\Gamma (E)$ as a locally convex space and it is well known that the fine very strong topology does not turn it into a locally convex space if $M$ is non compact (cf.\ \cite[Proposition 4.7]{HS2016}). Hence whenever we speak of the locally convex space $\Gamma (E)$ we assume that $\Gamma(E)$ carries the compact-open $C^\infty$ topology (which is coarser then the fine very strong topology if $M$ is non-compact).  
\end{rem}

We prove now that the evaluation map on spaces of sections is a smooth submersion, i.e.\ that it is locally conjugate to a projection (cf.\ \cite{Glofun} for more information on submersions between infinite-dimensional manifolds).

\begin{lem}\label{lem: sect:ev}
Let $q \colon Y\rightarrow X$ be a smooth submersion of finite-dimensional manifolds, then the evaluation $\ev \colon S_q (X,Y) \times X \rightarrow Y, \quad (s,x)\mapsto s(x)$ is a smooth submersion.
\end{lem}

\begin{proof}
Due to \cite[Corollary 11.7]{michor1980} the evaluation map $E\colon  \CfS (X,Y) \times X \rightarrow Y,\ (f,x)\mapsto f(x)$ is smooth, whence \ref{setup: sections} entails that its restriction $\ev$ is smooth on $S_q(X,Y)$.
We are left to prove the submersion property. Since $Y$ is finite-dimensional \cite[Theorem A]{Glofun} shows that $\ev$ will be a submersion if for every $(s,x) \in S_q (X,Y) \times X$ the tangent map $T_{(s,x)} \ev$ is surjective.
Consider  $v_x \in T_xC^\infty (\ast, X) \cong T_xX$ and $Y_s \in T_s C^\infty (X,Y)$ and remember that the latter tangent space is isomorphic to $$\{ Y \in C^\infty(X, TY) \mid Tq (Y) = s \text{ and } \exists K \subseteq X \text{ compact }, Y|_{X\setminus K} \equiv 0\} \cong \Gamma_c (s^* TY),$$ where $s^*TY$ is the pullback bundle.
Combining these identifications with \cite[Proof of Corollary 11.7]{michor1980} and \cite[Corollary 11.6]{michor1980} we obtain a formula for the tangent map
$$T_{(s,x)} \ev (Y_s,v_x) = Ts (v_x)+Y_s (x) \in T_{s(x)} Y.$$
Evaluating in $v_x = 0_x$, $T_{(s,x)} \ev (Y_s,0_x) = Y_s (x)$ is surjective by \cite[Lemma 11.8]{MR1930277}.
\end{proof}

 The proof of Lemma \ref{lem: sect:ev} used extensively that $E$ is a finite-dimensional vector bundle. A version for infinite-dimensional bundles (but compact $M$ can be found in \cite{MR3573833}. Following \cite[Corollary 2.10 and Corollary 2.9]{MR3573833}, we obtain.

 \begin{cor}\label{cor: sub:fin}
 Let $M,N$ be finite dimensional manifolds and $\cG = (G\toto M)$ be a finite-dimensional Lie groupoid, then
 $E \colon C^\infty (M,N) \times M \rightarrow N, E(f,m)
\coloneq f(m)$ and $\ev \colon \Bis (\cG) \times M \rightarrow G, \ev(\sigma,m) = \sigma (m)$ are smooth submersions.
 \end{cor}

\addcontentsline{toc}{section}{References}
\bibliography{monoid}

\begin{thebibliography}{ACMM89}
\providecommand{\url}[1]{\texttt{#1}}
\providecommand{\urlprefix}{URL }
\expandafter\ifx\csname urlstyle\endcsname\relax
  \providecommand{\doi}[1]{doi:\discretionary{}{}{}#1}\else
  \providecommand{\doi}{doi:\discretionary{}{}{}\begingroup
  \urlstyle{rm}\Url}\fi
\providecommand{\eprint}[2][]{\url{#2}}

\bibitem[ACMM89]{MR1040392}
Abbati, M.~C., Cirelli, R., Mani\`a, A. and Michor, P.
\newblock \emph{The {L}ie group of automorphisms of a principal bundle}.
\newblock J. Geom. Phys. \textbf{6} (1989)(2):215--235

\bibitem[ALR07]{MR2359514}
Adem, A., Leida, J. and Ruan, Y.
\newblock \emph{Orbifolds and stringy topology}, \emph{Cambridge Tracts in
  Mathematics}, vol. 171 (Cambridge University Press, Cambridge, 2007)

\bibitem[Ami18]{Amiri2017}
Amiri, H.
\newblock \emph{A group of continuous self-maps on a topological groupoid}.
\newblock Semigroup Forum \textbf{96} (2018)(1):69--80

\bibitem[AS15]{alas2012}
Alzaareer, H. and Schmeding, A.
\newblock \emph{Differentiable mappings on products with different degrees of
  differentiability in the two factors}.
\newblock Expo. Math. \textbf{33} (2015)(2):184--222

\bibitem[AS17]{AS17}
Amiri, H. and Schmeding, A.
\newblock \emph{{A differentiable monoid of smooth maps on Lie groupoids}}
  2017.
\newblock \eprint{arXiv:1706.04816v1}

\bibitem[Bas64]{bastiani}
Bastiani, A.
\newblock \emph{Applications diff\'erentiables et vari\'et\'es
  diff\'erentiables de dimension infinie}.
\newblock J. Analyse Math. \textbf{13} (1964):1--114

\bibitem[BGN04]{BGN04}
Bertram, W., Gl{{\"o}}ckner, H. and Neeb, K.-H.
\newblock \emph{Differential calculus over general base fields and rings}.
\newblock Expo. Math. \textbf{22} (2004)(3):213--282

\bibitem[Bko65]{MR0177002}
Bkouche, R.
\newblock \emph{Id\'eaux mous d'un anneau commutatif. {A}pplications aux
  anneaux de fonctions}.
\newblock C. R. Acad. Sci. Paris \textbf{260} (1965):6496--6498

\bibitem[Bos07]{MR2343354}
Bos, R.
\newblock \emph{Geometric quantization of {H}amiltonian actions of {L}ie
  algebroids and {L}ie groupoids}.
\newblock Int. J. Geom. Methods Mod. Phys. \textbf{4} (2007)(3):389--436

\bibitem[Bos11]{MR2844451}
Bos, R.
\newblock \emph{Continuous representations of groupoids}.
\newblock Houston J. Math. \textbf{37} (2011)(3):807--844

\bibitem[CSS15]{MR3318255}
Crainic, M., Salazar, M.~A. and Struchiner, I.
\newblock \emph{Multiplicative forms and {S}pencer operators}.
\newblock Math. Z. \textbf{279} (2015)(3-4):939--979

\bibitem[GGV77]{GGV77}
Gelfand, I., Graev, M. and Ver\u{s}ik, A.
\newblock \emph{Representation of a group of smooth mappings of a manifold $x$
  into a compact lie group}.
\newblock Compositio Mathematica \textbf{35} (1977):299--334

\bibitem[Gl{\"{o}}02]{hg2002a}
Gl{\"{o}}ckner, H.
\newblock \emph{{Infinite-dimensional Lie groups without completeness
  restrictions}}.
\newblock In A.~Strasburger, J.~Hilgert, K.~Neeb and W.~Wojty\'{n}ski (Eds.),
  \emph{{Geometry and Analysis on Lie Groups}}, \emph{Banach Center
  Publication}, vol.~55, pp. 43--59 (Warsaw, 2002)

\bibitem[Gl{\"o}16]{Glofun}
Gl{\"o}ckner, H.
\newblock \emph{{Fundamentals of submersions and immersions between
  infinite-dimensional manifolds}} 2016.
\newblock \eprint{arXiv:1502.05795v4}

\bibitem[Gra05]{MR2161810}
Grabowski, J.
\newblock \emph{Isomorphisms of algebras of smooth functions revisited}.
\newblock Arch. Math. (Basel) \textbf{85} (2005)(2):190--196

\bibitem[GS66]{MR0211421}
Guillemin, V. and Sternberg, S.
\newblock \emph{Deformation theory of pseudogroup structures}.
\newblock Mem. Amer. Math. Soc. No. \textbf{64} (1966):80

\bibitem[GSM17]{MR3696590}
Gracia-Saz, A. and Mehta, R.~A.
\newblock \emph{{$\mathcal{VB}$}-groupoids and representation theory of {L}ie
  groupoids}.
\newblock J. Symplectic Geom. \textbf{15} (2017)(3):741--783

\bibitem[HS17]{HS2016}
Hjelle, E.~O. and Schmeding, A.
\newblock \emph{Strong topologies for spaces of smooth maps with
  infinite-dimensional target}.
\newblock Expo. Math. \textbf{35} (2017)(1):13--53

\bibitem[KM97]{conv1997}
Kriegl, A. and Michor, P.
\newblock \emph{{The Convenient Setting of Global Analysis}}.
\newblock Mathematical Surveys and Monographs 53 (Amer. Math. Soc., Providence
  R.I., 1997)

\bibitem[Kos76]{MR0438405}
Kosmann, Y.
\newblock \emph{On {L}ie transformation groups and the covariance of
  differential operators}  (1976):75--89. Mathematical Phys. and Appl. Math.,
  Vol. 3

\bibitem[KSM02]{MR1958838}
Kosmann-Schwarzbach, Y. and Mackenzie, K. C.~H.
\newblock \emph{Differential operators and actions of {L}ie algebroids}.
\newblock In \emph{Quantization, {P}oisson brackets and beyond ({M}anchester,
  2001)}, \emph{Contemp. Math.}, vol. 315, pp. 213--233 (Amer. Math. Soc.,
  Providence, RI, 2002)

\bibitem[Lor09]{Lorenz09}
Lorenz, A.
\newblock \emph{Jet Groupoids, Natural Bundles and the Vessiot Equivalence
  Method}.
\newblock Ph.D. thesis, RWTH Aachen 2009.
\newblock \eprint{urn:nbn:de:hbz:82-opus-27635}

\bibitem[Mac05]{Mackenzie05}
Mackenzie, K. C.~H.
\newblock \emph{General theory of {L}ie groupoids and {L}ie algebroids},
  \emph{London Mathematical Society Lecture Note Series}, vol. 213 (Cambridge
  University Press, Cambridge, 2005)

\bibitem[Mic80]{michor1980}
Michor, P.~W.
\newblock \emph{Manifolds of {D}ifferentiable {M}appings}, \emph{Shiva
  Mathematics Series}, vol.~3 (Shiva Publishing Ltd., Nantwich, 1980)

\bibitem[Mic83]{MR702720}
Michor, P.
\newblock \emph{Manifolds of smooth maps. {IV}. {T}heorem of de {R}ham}.
\newblock Cahiers Topologie G\'eom. Diff\'erentielle \textbf{24}
  (1983)(1):57--86

\bibitem[Mrc05]{MR2159792}
Mrcun, J.
\newblock \emph{On isomorphisms of algebras of smooth functions}.
\newblock Proc. Amer. Math. Soc. \textbf{133} (2005)(10):3109--3113

\bibitem[Nee06]{neeb2006}
Neeb, K.-H.
\newblock \emph{Towards a {L}ie theory of locally convex groups}.
\newblock Jpn. J. Math. \textbf{1} (2006)(2):291--468

\bibitem[Nee08]{MR2501579}
Neeb, K.-H.
\newblock \emph{Lie group extensions associated to projective modules of
  continuous inverse algebras}.
\newblock Arch. Math. (Brno) \textbf{44} (2008)(5):465--489

\bibitem[Nee10]{MR2719276}
Neeb, K.-H.
\newblock \emph{On differentiable vectors for representations of infinite
  dimensional {L}ie groups}.
\newblock J. Funct. Anal. \textbf{259} (2010)(11):2814--2855

\bibitem[Nes03]{MR1930277}
Nestruev, J.
\newblock \emph{Smooth manifolds and observables}, \emph{Graduate Texts in
  Mathematics}, vol. 220 (Springer-Verlag, New York, 2003).
\newblock Joint work of A. M. Astashov, A. B. Bocharov, S. V. Duzhin, A. B.
  Sossinsky, A. M. Vinogradov and M. M. Vinogradov,

\bibitem[Ren80]{MR584266}
Renault, J.
\newblock \emph{A groupoid approach to {$C^{\ast} $}-algebras}, \emph{Lecture
  Notes in Mathematics}, vol. 793 (Springer, Berlin, 1980)

\bibitem[Ryb02]{MR1943713}
Rybicki, T.
\newblock \emph{A {L}ie group structure on strict groups}.
\newblock Publ. Math. Debrecen \textbf{61} (2002)(3-4):533--548

\bibitem[Sal13]{Salazar13}
Salazar, M.~A.
\newblock \emph{{Pfaffian groupoids}}.
\newblock Ph.D. thesis, Utrecht University 2013.
\newblock Cf.\ \url{http://arxiv.org/pdf/1306.1164v2}

\bibitem[SW15]{MR3351079}
Schmeding, A. and Wockel, C.
\newblock \emph{The {L}ie group of bisections of a {L}ie groupoid}.
\newblock Ann. Global Anal. Geom. \textbf{48} (2015)(1):87--123

\bibitem[SW16a]{MR3569066}
Schmeding, A. and Wockel, C.
\newblock \emph{Functorial aspects of the reconstruction of {L}ie groupoids
  from their bisections}.
\newblock J. Aust. Math. Soc. \textbf{101} (2016)(2):253--276

\bibitem[SW16b]{MR3573833}
Schmeding, A. and Wockel, C.
\newblock \emph{({R}e)constructing {L}ie groupoids from their bisections and
  applications to prequantisation}.
\newblock Differential Geom. Appl. \textbf{49} (2016):227--276

\bibitem[Wes67]{MR0211351}
Westman, J.~J.
\newblock \emph{Locally trivial {$C^{r}$} groupoids and their representations}.
\newblock Pacific J. Math. \textbf{20} (1967):339--349

\bibitem[Yud16]{Yudi16}
Yudilevich, O.
\newblock \emph{Lie Pseudogroups \`{a} la Cartan from a Modern Perspective}.
\newblock Ph.D. thesis, Utrecht University 2016.
\newblock \urlprefix\url{https://dspace.library.uu.nl/handle/1874/339516}

\bibitem[ZCZ09]{ZHUO}
Zhuo~Chen, Z.-J.~L. and Zhong, D.-S.
\newblock \emph{On the existence of global bisections of lie groupoids}.
\newblock Acta Mathematica Sinica, English Series \textbf{25}
  (2009)(6):1001--1014

\end{thebibliography}

\end{document}